\documentclass[smallextended,envcountsame]{svjour3}
\smartqed

\usepackage{amsmath,amssymb}
\usepackage{paralist}
\usepackage[T1]{fontenc}
\usepackage{microtype}
\usepackage{hyperref}
\usepackage{tikz,pgfplots}
\usepackage{graphicx}
\usepackage{subcaption}
\usepackage{enumitem}
\usepackage{nicefrac}
\usepackage{macros_perso}

\journalname{Set-Valued Var.\ Anal.}

\title{A new comparison principle for discrete Volterra equations with an application to convex sweeping processes with infinite delays}
\titlerunning{Sweeping processes with infinite delays}
\author{Thierno Mamadou Bald\'e \and Vuk Milisic \and Steffen Plunder}
\authorrunning{T. M. Bald\'e, V. Milisic, S. Plunder}
\institute{T.\,M.~Bald\'e \at Laboratoire de Math\'ematiques de Bretagne Atlantique,
Univ Brest, CNRS UMR 6205,
6, avenue Victor Le Gorgeu,
29200 Brest, France
\and V.~Milisic \at Laboratoire de Math\'ematiques de Bretagne Atlantique,
Univ Brest, CNRS UMR 6205,
6, avenue Victor Le Gorgeu,
29200 Brest, France
\and S.~Plunder \at Institute for Advanced Study of Human Biology (WPI-ASHBi), Kyoto University, Kyoto, 606-8303, Japan}
\newcommand{\kernel}{\varrho}
\newcommand{\tkernel}{\tilde{\varrho}}
\newcommand{\tp}{\tilde{p}}
\renewcommand{\tr}{R}
\newcommand{\tP}{\tilde{P}}
\newcommand{\tf}{\tilde{f}}
\DeclareMathOperator*{\argmin}{argmin}

\begin{document}

\maketitle

\begin{abstract}
Comparison principles for Volterra equations play a role analogous to maximum principles in PDEs: 
they provide positivity and stability information on the solution and allow one to control
the output of bounded inputs. In the continuous setting, such results often rely on Laplace-transform 
or spectral methods (see e.g. Gripenberg et al. \cite[Ch. 4. \& 7.]{grip}). However, these tools are not 
uniform in the discretization step  $h$ hence fail in discrete or semi-discrete approximations. 
The present note introduces a resolvent-free argument yielding uniform $L^\infty(0,T)$-bounds
for non-negative kernels.

Compactness is a key ingredient in order to show existence of sweeping processes.
While in the classical framework it is well established, adding an infinite 
distribution of delays complicates greatly the obtaining of such a result.
In a first step we show a general energy decay estimate, which is then used to
establish compactness. The argument is carried out in the discrete setting
and that necessitates the introduction of the  new comparison principle.

In the classical sweeping process the previous position of the particle
lies on the boundary of the constraint set, staying $O(h)$
close to the next projection point ($h$ is the discretization step). Our delay model projects the particle's
averaged (by a unit measure kernel) past positions  to the constraint set.
Numerical simulations show that the projected point can lie at $O(1)$
distance from the convex set's boundary.
\end{abstract}

\tableofcontents

\section{Introduction}

Volterra equations arise naturally in the modeling of systems with memory, where the present state depends
on the past history through a convolution kernel. In this work, we provide tools for analyzing a novel type of Volterra 
equations which is motivated by a problem from modelling of multi-cellular systems with adhesive memory and non-overlap 
conditions (see \cite{venel08} for the classical differential inclusion setting). 

Cells with adhesive memory have been modelled extensively using Volterra-type equations where the memory 
represents the forces of adhesion to past positions; 
see e.g.~\cite{OelzSch10,PreVi,Manhart2015,MiOel.1,MiOel.2,MiOel.3,MiOel.4} and Figure~\ref{fig.cells.example} (left). 
The system modelling the quasi-steady state of cells at each point in time reads 
\begin{align*}
  z(t) - \int_0^t \kernel(t-s) z(s)\,ds = f(t), \quad t > 0,
\end{align*}
where $z(t)$ represents the position of the cell at time $t$, $\kernel$ is a 
nonnegative kernel describing the adhesive memory, and $f(t)$ is an external force acting on the cell.

\begin{figure}[h]
	\centering
	\includegraphics[width=0.5\textwidth]{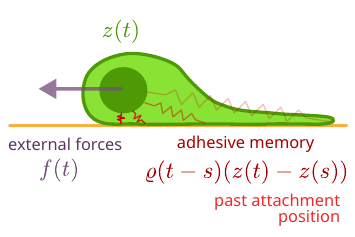}
	\includegraphics[width=0.4\textwidth,trim=60 60 60 60,clip]{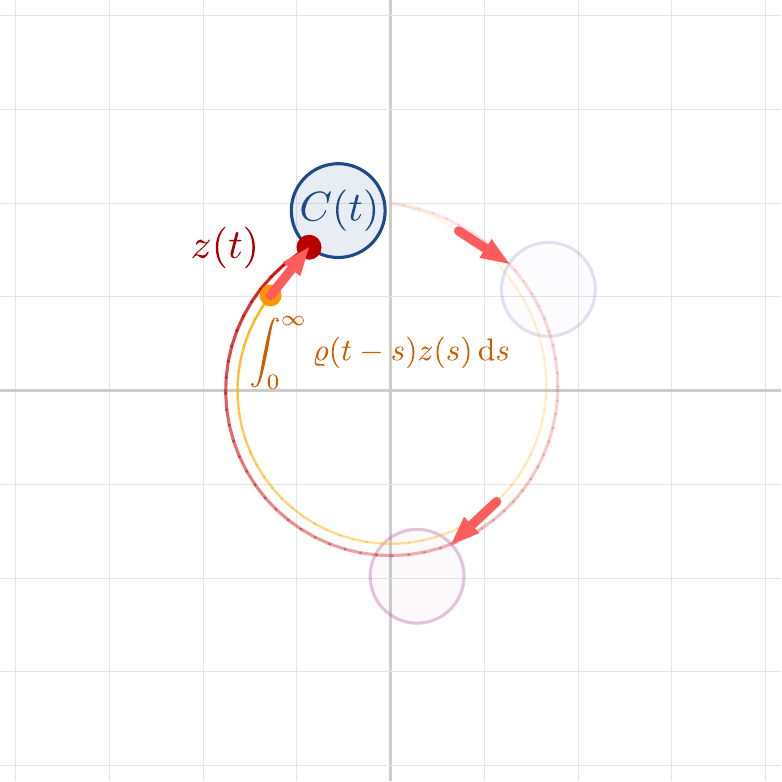}
	\caption{(left) Adhesive memory cell dynamics. (right) Dynamics of a constraint set moving in circles. The projection distance (red arrow) is not infinitesimal, which is an inherent property of the system and not a time-stepping artifact.}
	\label{fig.cells.example}
\end{figure}

When multiple cells are considered, one might want to impose non-overlapping constraints between cells, effectively constraining the state to a feasible set $C(t)$ at each time~$t$ (see Figure~\ref{fig.cells.example}, right, for an example with a moving circular constraint).
This leads to a new Volterra-type sweeping process: find $z(t) \in C(t)$ for $t>0$, such that
$$
z(t) - \int_0^t \kernel(t-s) z(s)\,ds - f(t) \in -N_{C(t)}(z(t)),\quad t > 0,
$$
where $C(t)$ can be for instance the 
interior convex approximation of the non-overlapping constraint \cite{venel08}
and simplified to a time-dependent convex set and $N_{C(t)}(z(t))$ is the normal cone to $C(t)$ at $z(t)$.

As the movement of the set $C(t)$ might be much faster than the natural dynamics of the unconstrained 
system, we experience here a challenging geometric setup. This becomes visible when the dynamics are 
recast as a projective system, such as 
\begin{align*}
  z(t) = P_{C(t)}\Big( \int_0^t \kernel	(t-s) z(s)\,ds + f(t) \Big), \quad t > 0.
\end{align*}
Here the projection distance depends crucially on the kernel and on the past trajectory $z(t)$.
This shows the fundamental difference with  classical sweeping processes. Indeed, for Moreau's model,
  the projection distance is infinitesimal due to the presence of a time-derivative term.
Dealing with non-infinitesimal projection distances (as illustrated in Figure~\ref{fig.cells.example} (right)) requires new bounds on the solution $z(t)$ in order to 
establish stability and compactness. Often  stability results 
are based on some kind of Gronwall's lemma (for ODEs for instance) or more generally on comparison principles.
For Volterra equations these are not straightforward to obtain, this is due to the non-local nature of the convolution operator
and from the physical perspective, the memory effect implies that the solution at time $t$ depends on the entire past history,
and can exhibit oscillations or growth that are not present in local differential equations.
Comparison principles for such equations play a 
role analogous to maximum principles in parabolic PDEs: they ensure that bounded or positive data produce 
bounded or positive solutions, and they provide essential \emph{a~priori} bounds for qualitative and numerical analysis.
 
In the continuous setting, stability results can often be obtained  thanks to
the concept of a {\em resolvent}  associated with the Volterra kernel. This is a special function $r$ that solves:
\[r(t) - \int_0^t \kernel	(t-s) r(s)\,ds = \kernel(t), \qquad t > 0.\]
Using the Laplace transform  provides a spectral
decomposition of $r$  as a sum of exponentials and a remainder term that
is often integrable or decays in a sufficient manner \cite[Ch.~4 \& 7]{grip}. 
However, in many applications—especially when the kernel is not exponentially decaying or when it depends on 
the discretization parameter—such a spectral decomposition is not directly available. Moreover, 
even if Wiener’s lemma can be invoked to provide a similar asymptotic description of the discrete resolvent, 
the corresponding representation is not uniform with respect to the discretization step $h$ 
and therefore does not provide uniform \emph{a~priori} bounds for the discrete solutions.

Classical convergence and stability analyses of discrete Volterra equations, 
such as those of McKee and Jones~\cite{JonesMcKee1982}, Bakke and Jackiewicz~\cite{BakkeJackiewicz1988}, 
and Lubich~\cite{Lu.88.1,Lu.88.2}, rely on spectral or Laplace-transform techniques that 
require exponentially decaying kernels or  j whose Laplace transform is well-behaved. 
These approaches provide powerful tools 
for these kernels but fail to yield uniform bounds when the memory kernel is not explicit 
or lacks analyticity. The present work develops a comparison principle that 
circumvents the use of any spectral decomposition and applies equally to continuous and 
discrete settings. The key novelty lies in the construction of an {\em initial layer} 
corrector compensating errors introduced by the tails of the leading order part of the super-solution.
Applying the Volterra operator to the complete ansatz (leading order plus initial layer corrector) is then able to 
dominate the right-hand side of the Volterra equation, 
whose solution we aim to control, thanks to a comparison principle specific to Volterra equations \cite{grip,MiOel.1}.
We underline that the comparison principle found in~\cite[Theorem~2.7]{MiOel.1} does not hold in the case of a 
non-exponential kernel that is constant in time.

Moreau's sweeping process (classical formulation) 
is the baseline model: a first-order differential inclusion 
  $-\dot X(t)\in N_{C(t)}(X(t))$ introduced by J.-J. Moreau \cite{Moreau1999}.
  Delay and history-dependent perturbations have been studied since: 
  existence results for sweeping processes with delays and history operators were 
  established in works by Castaing \& Monteiro Marques \cite{Castaing1997} and 
  Edmond \cite{Edmond2006}.
 Most results for delayed sweeping processes keep a time-derivative
 term (i.e. explicit $\dot X$ or its fractional analogue) in the formulation; 
 the analysis relies on this differential feature to obtain compactness and pass-to-the-limit.
Fractional-time derivative sweeping processes (Caputo-type) have been developed recently; 
  see e.g. Zeng \cite{Zeng2023} and Bouach \cite{Bouach2025}, 
  by a fractional derivative and yield existence/uniqueness 
   and—importantly—compactness properties coming from fractional regularizing effects.
By contrast, replacing the time-derivative by a general Volterra (integro-differential / history) 
operator leads to different analytic difficulties: compactness is generally lost unless very 
strong hypotheses are imposed. 
Recent papers prove well-posedness for Volterra-type integro-differential 
sweeping processes using tailored a priori estimates; 
see Vilches \cite{Vilches2024}, Godoy et al. \cite{Godoy2024}, and Haddad \cite{Haddad2025}.
State-dependent (moving set depends on $x$) and prox-regular (non-convex) extensions have been 
developed for both delay and fractional setups; proofs typically combine prox-regular geometry 
with fixed-point or reparametrization techniques.
Numerical and applied analyses (contact, frictionless impact, plasticity) 
motivate the integro-differential and fractional variants.

Our work contributes to this recent line of research by establishing a new comparison principle
for discrete integral Volterra equations with non-exponential non-negative kernels. 
To be more precise, our problem does not involve a time-derivative term,
rather it is replaced by a memory operator involving an infinite distribution of delays.
This operator has been introduced in \cite{OelzSch10} to account
for a microscopic friction mechanism mediated by transient elastic linkages to past positions.
Biological considerations motivate this model, as cells adhere to their past positions
through transient bonds that form and break over time.
In order to show existence of sweeping processes relying on  such principles, 
we use minimizing movements à la De Giorgi. The key step is to obtain
uniform $H^1$-bounds on the discrete solutions, which then provides
compactness by Ascoli-Arzelà's theorem. The new comparison principle
allows one to obtain such uniform bounds in the discrete setting.
This is crucial for passing to the limit and establishing existence of solutions.
To our knowledge, there is no such result in the literature yet,
even for the case of the simplest sweeping process where the moving convex set
is a sphere.

In order to illustrate the originality of our model, we present in Figures \ref{fig.simulations} 
a numerical simulation of the delayed sweeping process with a moving circular constraint following a Lissajous curve. 
In orange we plot the trajectory of the mean position mediated by linkages.
The projection distance (red arrow) is not infinitesimal, which is an inherent property of the system and not a 
time-stepping artifact. The red curve is the trajectory of the projection of the mean position onto the constraint set.
In order to show how complex the dynamics can be, we present in Figure
\ref{fig.simulation2} a simulation of the same system with a non-circular, rotating, stadium shaped constraint.
We plot in blue the classical sweeping process with a time-derivative term, and in red the delayed sweeping process.
\begin{figure}[h]
	\centering
	\begin{minipage}{0.45\textwidth}
		\centering
		\includegraphics[width=\textwidth,trim=60 60 60 60,clip]{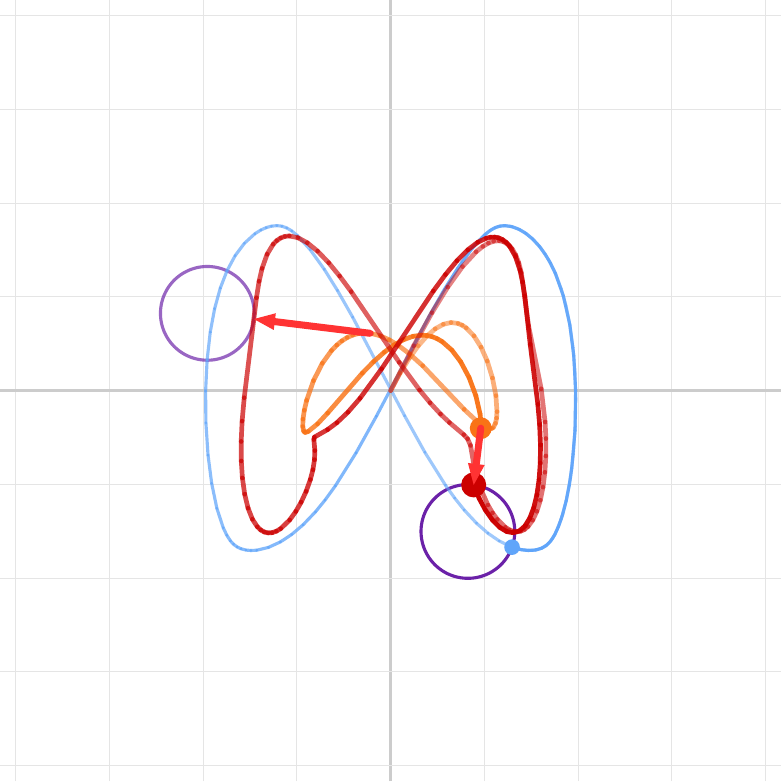}
		\subcaption{Simulation of the delayed sweeping process with a moving circular constraint following a 
		Lissajous curve.}
		\label{fig.simulation1}
	\end{minipage}
	\hfill
	\begin{minipage}{0.45\textwidth}
		\centering
		\includegraphics[width=\textwidth,trim=60 60 60 60,clip]{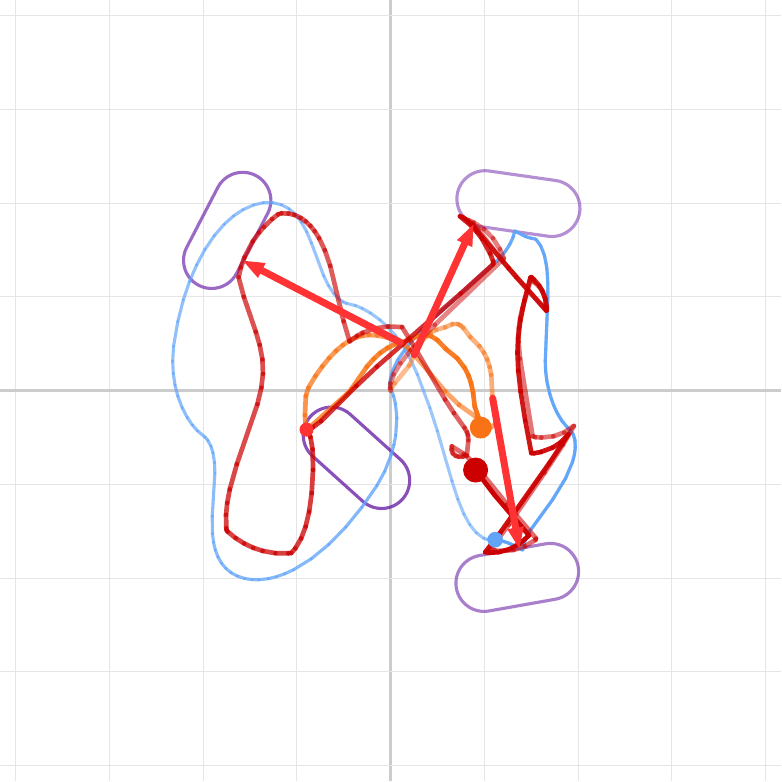}
		\subcaption{Simulation of the delayed sweeping process with a moving stadium-shaped constraint 
		following a Lissajous curve.}
		\label{fig.simulation2}
	\end{minipage}
	\caption{Simulations of the delayed sweeping process with moving constraints. The projection distance (red arrow) is not infinitesimal, which is an inherent property of the system and not a time-stepping artifact.}\label{fig.simulations}
\end{figure}

The paper is organized as follows. 
In Section \ref{sec:assumptions}, we gather all the assumptions and state the main results.
Section~\ref{sec:continuous} recalls the 
continuous framework and basic assumptions ensuring the existence of a unique solution. 
In Section~\ref{sec:initiallayer}, we study a canonical auxiliary problem that reveals 
the structure of the \emph{initial layer} associated with the convolution kernel. 
Building on this analysis, Section~\ref{sec:comparison} establishes a new comparison 
principle valid for non-exponential kernels. 
Section~\ref{sec:discrete} introduces a discrete counterpart of the continuous model, 
derives a discrete conservation relation, and proves uniform $L^\infty$ bounds independent 
of the discretization step $h$. Section~\ref{sec:application} applies these results to a 
delayed sweeping process, establishing energy estimates and compactness properties needed 
for existence results. Then by standard arguments one passes to the limit in the discrete
approximations and recovers an existence result for the continuous delayed sweeping process.
Section~\ref{sec:numerics} describes the numerical algorithm for computing the discrete sweeping sequence,
including the projection via Newton's method, and presents numerical examples.
Finally, we detail conclusions and perspectives of this work. 
The Appendices \ref{sec:moments},  \ref{sec:discrete_convolution} and \ref{sec:decay_lipschitz} contain auxiliary technical results.

\section{Assumptions and main results}\label{sec:assumptions}

In this section, we gather the assumptions used 
throughout the paper and state the main results.
Assumptions concerning the Volterra equations that
we consider are stated hereafter.
\begin{hypo}[Data assumptions]\label{hypo.data}
Throughout this paper, we assume the following conditions on the data:
	\begin{compactenum}[i)]
		\item the source term  $f \in L^\infty (\rr)$,  
		\item the kernel $\kernel$ is non-negative unit measure and satisfies $\kernel \in L^1(\rr,(1+a)^2)\cap \Lip(\rr)$,
		\item the kernel satisfies the monotonicity condition: $\da \kernel (a) \leq 0$,
		\item there exists a constant $\zeta > 0$ such that $- \da \kernel(a) / \kernel(a) \leq \zeta$ for all $a>0$,
		\item the past condition reads: $\zp \in \Lip (\RR^-)$.
	\end{compactenum}
\end{hypo}

In the rest of the paper the renormalization of the kernel is made for
the sake of conciseness and all the results hold for generic non-negative
integrable kernels.
\renewcommand{\tkernel}{\kernel}

\noindent For the sweeping process, we require some regularity on the 
moving convex sets:
\begin{hypo}[Convex sets assumptions]\label{hypo.convex.main}
For the application to delayed sweeping processes, we consider a family of 
closed convex sets $(C^n)_{n\in\N}$ in $\mathbb{R}^d$, satisfying:
\begin{compactenum}[i)]
	\item  there exists $R>0$ such that for all $n\in\N$, $C^n \subset B(0,R)$,
	\item compactness of the sequence: for all $h<h_0$, 
	setting $N=\lfloor T/h \rfloor$, the sequence of convex sets $(C^n)_{0\leq n \leq N}$ satisfies the $H^1$-like estimate:
	\begin{equation}\label{eq.bv.c.main}
\sum_{n=1}^N d_H(C^{n},C^{n-1})^2/h \leq C < \infty,
\end{equation}
where the constant $C$ is uniform with respect to $h$,
\end{compactenum}
\end{hypo}

\subsection{Comparison principles for Volterra equations: continuous and discrete settings}
In a first step, and for the sake of clarity, we show the 
theoretical construction leading to the comparison principle 
in the continuous setting.
\begin{theorem}[Comparison principle for continuous Volterra equations]\label{thm.main.comparison}
	Assume that $z$ solves:
	\begin{equation}\label{eq.z.main}
		\left\{\begin{aligned}
			&	 \int_\rr (z(t)-z(t- a)) \kernel(a) da = f(t), & t>0 \\
			&	z(t) = \zp(t), & t<0,
		\end{aligned}\right.
	\end{equation}
	under Hypothesis \ref{hypo.data}. Then $z \in L^\infty (0,T)$ for any finite $T$, and the bound is independent of the spectral decomposition of the resolvent.
\end{theorem}
This result is made possible by introducing an initial layer corrector.
For the sake of clarity, technical aspects are further detailed in section \ref{sec:initiallayer}. 
\noindent Mimicking the continuous setting, we now turn to the discrete framework.
We discretize the time interval $[0,T]$ with a step $h>0$ and
 we denote $t_n:= n h$ for all $n \in \{0, \ldots, N\}$ where $N:= T/h$. 
 We assume for the sake of simplicity that $T/h \in \N$.
 Then we use piecewise constant approximations. 
 Within this framework, we state the following corresponding result.
\begin{theorem}[Uniform discrete comparison principle]\label{thm.comp.disc.main}
	Consider the discrete problem:
	\begin{equation}\label{eq.disc.z.main}
	\left\{\begin{aligned}
		&	Z^n - h \sum_{j\geq 0} Z^{n-j} \tr_j = f^n , & n \in \{0, \ldots, N\} \\
		&	Z^n = Z_p^n , & n<0,
	\end{aligned}\right.
	\end{equation}
	where $R_j:= \frac{1}{h} \int_{jh}^{(j+1)h} \kernel(a) da$.
	Assume that 
	Hypotheses \ref{hypo.data} holds. Then the solution $(Z^n)_{n\in \{0,\dots,N\}}$ of \eqref{eq.disc.z.main} is bounded uniformly with respect to $h$.
\end{theorem}
\noindent Again the key ingredient is the discrete initial layer corrector and its 
control (cf section \ref{subsec:discrete}).

\subsection{Energy estimates for delayed sweeping processes}

We formulate our sweeping process as the projection of the averaged 
past position onto the convex set:
$$
X^n:= P_{C^n} (\overline{X}^n), \quad 
\overline{X}^n :=  \frac{h \sum_{j\geq 1} X^{n-j} \tr_j}{h \sum_{j\geq 1} \tr_j}, \quad n \in \{0, \ldots, N\}.	
$$
where the kernel's weights $(\tr_j)_{j\in \N}$ are defined as above.
Then we show that this sweeping process can be equivalently reformulated as the following
delayed gradient flow \cite{Mi.5} and that it satisfies the energy dissipation estimates:
\begin{proposition}[Energy dissipation]\label{prop.nrj.main}
\begin{equation}\label{eq.sweeping.main}
	 X^n := \argmin_{W \in C^n} \cE_n(W), \quad \cE_n(W) := \frac{h}{2}\sum_{j\geq1} (W-X^{n-j})^2 \tr_j.
	\end{equation}
	Assume that $(\tr_j)_{j\in \N}$ is decreasing and Hypothesis \ref{hypo.convex.main} holds. 
	Then:	
	\[
	\cE_{N}(X^{N}) +  \sum_{n=0}^{N-1} D_n \le e^{2 T} \cE_{0}(X^0) + e^{2 T} \sum_{n=0}^{N-1} e(C^n,C^{n+1})^2/h,
	\]
	where $D_n= \frac{h}{2} \sum_{j\ge1} (\tr_{j}-\tr_{j+1}) | X^{n} - X^{n-j}|^2 \geq 0$ is the dissipation term.
\end{proposition}



\subsection{Convergence to continuous delayed sweeping process}
Define the piecewise linear interpolant $X_h$ of the discrete solution $X^n$ on $(0,T)$.
In the same way, define the piecewise constant interpolant $\overline{X}_h$ of the discrete averages $\overline{X}^n$ on $(0,T)$.
Then, we have the following convergence result. 
\begin{theorem}[Limit of the discrete sweeping process]\label{thm:limit_sweeping.main}
	Under Hypotheses \ref{hypo.data} and \ref{hypo.convex.main}, assume the uniform estimate
	$ \| X_h \|_{H^1(0,T)}\le C$ for all $h>0$, for any finite $T>0$.
	
	Then, as $h\to0$, there exist $X\in H^1([0,T];\mathbb R^m)$ and a subsequence such that 
	$X_h \to X$ in $C^0 (0,T;\mathbb R^m)$, and the discrete averages converge 
	$\overline{X}_h \to \overline{X}$ in $L^p(0,T;\mathbb R^m)$ for $1\le p<\infty$, where
	\[
	\overline{X}(t) := \int_0^\infty \varrho(s)\,X(t-s)\,ds.
	\]
	Moreover, the limit function $X$ satisfies the Moreau-type integral inclusion:
	\begin{equation}\label{eq:limit_inclusion.main}
		\overline{X}(t)-X(t)\in N_{C(t)}\big(X(t)\big), \quad\text{for a.e. }t\in(0,T),
	\end{equation}
	equivalently, $X(t)=P_{C(t)}\big(\overline{X}(t)\big)$ for a.e. $t\in(0,T)$.
\end{theorem}

\begin{remark}
	The general convergence result stated in Theorem \ref{thm:limit_sweeping.main} 
	requires an $H^1(0,T)$ bound on the discrete solution, which is difficult to 
	establish for general convex sets. The following result provides such a bound 
	for the special case of moving circular sets. In order to obtain these 
	uniform $H^1$ bounds, we crucially rely on the new comparison principle
	presented above.
\end{remark}

\begin{theorem}[Convergence for moving circular sets]\label{thm:circular.main}
	Consider the case where $C(t) := \overline{B}(c(t), R)$ with $c \in H^1([0,T];\mathbb{R}^m)$ and $R \in \mathbb{R}_+$.
	Discretize the sets as $C^n := C(nh)$ for $n=0,\dots,N$ with $N = \lfloor T/h \rfloor$.
	Assume that $(\tr_j)_{j\geq 0}$ is non-increasing with mass one, and that $X^{-j} = X^{-1} \in C_0$ for all $j\geq 2$.
	Moreover we suppose that the discrete kernel satisfies the following estimate:
	\begin{equation}\label{eq.kernel.estimate.circular}
		h \sum_{j\geq 1} \left( \frac{R_{j-1}-R_j}{h}\right)^2 \frac{1}{R_j} \leq C_R < \infty,
	\end{equation}
	where the constant $C$ is independent of $h$.

	Then the discrete sweeping sequence $(X^n)_{n\ge0}$ defined by \eqref{eq.sweeping} satisfies the uniform bounds:
	\[
	\sup_{n\in \{0,\dots,N\}}|X^n|\le C,\qquad
	\left\| \dt X_h \right\|^2_2 := \frac{1}{h}\sum_{n=0}^{N-1} |X^{n+1}-X^n|^2  \le C,
	\]
	where $X_{h}$ is the piecewise-linear interpolant of $X^n$ on $(0,T)$, and $C>0$ is independent of $h$.
	
	Consequently, by Theorem \ref{thm:limit_sweeping.main}, the discrete solutions converge to a continuous delayed sweeping process satisfying $X(t)=P_{C(t)}\big(\overline{X}(t)\big)$ for a.e. $t\in(0,T)$.
\end{theorem}

\begin{remark}
	The kernel condition \eqref{eq.kernel.estimate.circular} 
	is satisfied for a large class of kernels, including those 
	with algebraic decay such as $\kernel(a) = (1+a)^{-\alpha}$ for $\alpha > 0$.
\end{remark}

\section{The continuous setting}\label{sec:continuous} 

We start with the problem: find $z \in L^1_\loc(\rr)$ solving  
\begin{equation}\label{eq.z}
	\left\{\begin{aligned}
& 	\int_\rr (z(t)- z(t-a)) \kernel(a) da = f(t), & t>0 \\
& 	z(t) = \zp(t), & t<0,
	\end{aligned}\right.
\end{equation}
where the data satisfy Assumptions \ref{hypo.data}. 
By standard arguments:
\begin{theorem}
	Under assumptions \ref{hypo.data}, there exists a unique $z \in L^1_\loc(\rr)$ solving \eqref{eq.z}.
\end{theorem}
\noindent The proof is classical and relies on the construction of a resolvent and  can be found for instance in \cite[Theorem 3.5, p. 44]{grip}.

The solution can be represented by the 
formula $z(t) = f(t) + \int_0^t r(t-a) f(a) da$ 
where $r$ is the resolvent associated with $\kernel$.
Then techniques based on the Laplace transform lead to 
a spectral decomposition of the resolvent giving 
sharp $L^p$ estimates of the solution in a direct way.
Our aim is to provide $L^\infty$-bounds on $z$ that do not rely on 
the knowledge of $r$.
Indeed, we expose here ideas that should be efficient at the discrete level
where the spectral decomposition does not provide uniform asymptotic characterization
with respect to $h$. 
Our approach relies on the concept of a comparison principle.
For the sake of clarity, we detail the basic ideas first in the continuous setting and then
transpose these to the discrete level.
For this purpose, we first study an initial layer corrector.

Some of the arguments presented in this section are formal. This is done for the sake of clarity and to expose the main ideas.
All the computations can be made rigorous using the concept of Lipschitz regularity along the characteristics of the transport
equation (see for instance the detailed computations of \cite[Lemma 3.1]{MiOel.1}). On the other hand,
since the discrete counterpart follows with full rigor, the presentation of the continuous setting was optimized 
in this sense.

\subsection{An initial layer}\label{sec:initiallayer}
We solve for the time being 
a given problem whose importance will be explained later on. 
We denote by $\tkernel$ the renormalized version of the convolution kernel above:
$\tkernel (a) := \kernel(a)/ \int_\rr \kernel(\tia) d\tia$.
\begin{equation}\label{eq.w}
		 w(t)-\int_0^t w(t-a) \tkernel(a) da = \int_t^\infty \left(a-t \right) \tkernel(a) da, \;  t>0.
\end{equation}
This problem can also be  reformulated as 
\begin{equation}\label{eq.L.conserved}
	\left\{
	\begin{aligned}
		& \int_\rr (w(t)-w(t-a))\kernel(a) da = 0, & t\geq 0,\\
		& w(t) = -t, & t< 0.\\
	\end{aligned}
	\right.
\end{equation}
Setting the elongation variable as $u(a,t) := w(t)-w(t-a)$, one can 
also reformulate the problem as a non-local PDE on this new variable 
\cite{MiOel.1}: find $u \in L^\infty((0,T)\times\rr,(1+a)^{-1})$ solving 

\begin{equation}\label{eq.elongation.u}
	\left\{
\begin{aligned}
	& \dt u + \da u = - \int_\rr \da \tkernel(\tia) u(\tia,t)d \tia, & a>0,\; t>0,\\
	& u(0,t)=0, & a=0,\; t>0,\\
	& u(a,0)= w(0^+) - w(-a)=:u_I(a), & a>0,\; t=0.
\end{aligned}
\right.
\end{equation}

Indeed, formally one has: $\dt u + \da u = \dot{w}$ by definition 
of $u$, then one writes:
\begin{align}
\ddt{} \int_\rr \tkernel(a) u(a,t) da &- \int_\rr \da \tkernel (a ) u(a,t) da = \left(\int_\rr \tkernel(a) da\right) \dot{w} = \dot{w}
\end{align}
but with this new definition \eqref{eq.L.conserved} becomes simply  
$\int_\rr \tkernel(a) u(a,t) da=0$ for a.e. $t>0$, this shows that 
$\dot{w} = - \int_\rr \da \tkernel(a) u(a,t) da$.  
Then using the first transport equation above shows how, starting from $w$ solving \eqref{eq.w},
one recovers $u$ solving \eqref{eq.elongation.u}.
Conversely, starting from $u$ solving \eqref{eq.elongation.u}, one recovers $w$ solving \eqref{eq.w} by setting 
$w(t) := w(0^+) + \int_0^t \dot{w}(s) ds$ with $\dot{w}(s) = - \int_\rr \da \tkernel(a) u(a,s) da$. 
Although these computations (and the rest of formal computations of this section) 
are formal, they can be made rigorous using the 
method of characteristics (see for instance the detailed computations of \cite[Lemma 3.1]{MiOel.1}).
A simple computation shows that $w(0^+)= \mu_1/\mu_0 = \mu_1$ where 
$\mu_k := \int_\rr a^k \kernel(a) da$.
Using \cite[Theorem 6.1]{MiOel.2}, there exists a unique solution 
$u \in L^\infty((0,T)\times\rr,(1+a)^{-1})$ of \eqref{eq.elongation.u}.

We are interested in a bound on $w$ so it seems natural to integrate 
\eqref{eq.elongation.u} in time. This leads to define 
$p(a,t) := \int_0^t u(a,s) ds$ which solves:
\begin{equation}\label{eq.p}
	\left\{
	\begin{aligned}
		& \dt p + \da p= - \int_\rr \da \tkernel(\tia) p(\tia,t)d \tia + u_I(a),& a>0,\; t>0,\\
		& p(0,t)=0, & a=0,\; t>0,\\
		& p(a,0)= 0, & a>0,\; t=0,
	\end{aligned}
	\right.
\end{equation}
we then make a change of variables to get rid of the $a$-dependent extra source term 
that appears on the right hand side. 
Setting $\tp(a,t) := p(a,t) - \int_0^a u_I(s)ds$, it solves:
 \begin{equation}\label{eq.tp}
 	\left\{
 	\begin{aligned}
 		& \dt \tp + \da \tp= - \int_\rr \da \tkernel(\tia) \tp(\tia,t)d \tia  - \int_\rr \da \tkernel  \int_0^a u_I(s) ds, & a>0,\; t>0\\
 		& \tp(0,t)=0 & a=0,\; t>0\\
 		& \tp(a,0)= - \int_0^a  u_I(s)ds & a>0,\; t=0
 	\end{aligned}
 	\right.
 \end{equation}
 Using then Proposition \ref{prop:decay_weighted}, the integration by parts formula is well defined and we have:
 $$
 \int_\rr \da \tkernel \int_0^a u_I(s)ds da = - \int_\rr \tkernel(a) u_I(a) da = 0.
 $$ 
 so that \eqref{eq.tp} reduces simply to: 
$$
\left\{
\begin{aligned}
	& \dt \tp + \da \tp= - \int_\rr \da \tkernel(\tia) \tp(\tia,t)d \tia ,& a>0,\; t>0\\
	& \tp(0,t)=0 & a=0,\; t>0\\
	& \tp(a,0)= - \int_0^a  u_I(s)ds & a>0,\; t=0
\end{aligned}
\right.
$$

\begin{theorem}\label{thm.w}
	Assume that $\da \kernel (a) \leq 0$ and that there exists a constant $\zeta > 0$ s.t. $- \da \kernel / \kernel < \zeta$, then 
	$w \in L^\infty (\rr)$ and $w \geq 0$ for a.e. $t\geq 0$.
\end{theorem}

\begin{proof}[Proof of Theorem \ref{thm.w}]
 	We apply the key idea of \cite[Lemma 5.1]{MiOel.2} and write:
 	\begin{align}
 	\ddt{} \int_\rr \tkernel (a) |\tp(a,t)| da &- \int_\rr \da \tkernel (a) |\tp(a,t) | da \leq - \int_\rr \da \tkernel (a) |\tp(a,t) | da,
 	\end{align}
 	providing that 
	\begin{align}
	\int_\rr \tkernel(a) | \tp(a,t) | da &\leq \int_\rr \tkernel(a) | \tp(a,0) | da 
	= \int_\rr \tkernel(a) \left| \int_0^a u_I\right| da \lesssim \| \tkernel (1+a)^2\|_{1}.
	\end{align}
 	Then writing that 
 	$$
 	\begin{aligned}
 		w(t) & = w(0) + \int_0^t \dot{w}(s)\,ds
		= w(0) - \int_0^t \int_\rr \da \tkernel(a) u(a,s)\, da\, ds  \\
		& = w(0) -  \int_\rr \da \tkernel(a) p(a,t)\, da,
 	\end{aligned}
 	$$
 	this provides that:
 	\begin{align}
 			\left|w(t)\right| &\leq \frac{\mu_1}{\mu_0} + \int_\rr \frac{-\da \tkernel}{\tkernel} \tkernel(a) | p(a,t) | da\leq \frac{\mu_1}{\mu_0} + \zeta \int_\rr   \tkernel(a) | p(a,t) | da \\
 			&\leq C + \zeta \int_\rr | \tp(a,t) | \tkernel(a) da +  \zeta \int_\rr \left| \int_0^a u_I(s)ds \right| \tkernel(a) da \\
			&\leq C \left(1 + \| \tkernel \|_{L^1(\rr,(1+a)^2)}\right).
 	\end{align}
The initial layer corrector $w$ is thus bounded in $L^\infty(\rr)$. Moreover it is non-negative since the source term in \eqref{eq.w} is non-negative and the kernel is also non-negative.
Indeed, the resolvent associated with \eqref{eq.w} is positive: 
as we assumed that $0\leq -\da \kernel(a) / \kernel(a) \leq \zeta$, 
 	which gives immediately that 
 	$$
 	\int_0^t \tkernel (a) da = 1 - \int_t^\infty \tkernel(a) da \leq 1 - \frac{\kernel(0)\exp(-\zeta t)}{\zeta \mu_0} < 1.
 	$$
 	We are in the position to  apply \cite[Proposition 9.8.1]{grip}, which ensures the positivity of the resolvent. Next, one has the 
	representation formula for the solution $w$ of \eqref{eq.w}: $w(t)= \int_t^\infty(a-t)\tkernel(a) da + \int_0^t r(t-a) \int_a^\infty (s-a)\tkernel(s) ds da\geq 0$.
	\end{proof}
 
 \subsection{A new comparison principle}\label{sec:comparison} 

\begin{theorem}\label{thm.comparison}
	Assume that $z \in L^1_\loc(\rr)$ solves:
	\begin{equation}\label{eq.z.comparison}
		\left\{\begin{aligned}
			&	 \int_\rr (z(t)-z(t- a)) \kernel(a) da = f(t), & t>0 \\
			&	z(t) = \zp(t), & t<0,
		\end{aligned}\right.
	\end{equation}
	under Assumptions \ref{hypo.data}. 
	then $z \in L^\infty (0,T)$ for any finite $T$.   
\end{theorem}
 
 \begin{proof}[Proof of Theorem \ref{thm.comparison} (stated in Section \ref{sec:assumptions} as Theorem \ref{thm.main.comparison})]
 	Our aim is now to construct a super-solution of our problem, for this purpose we start as in the proof of \cite[Theorem 1.1]{MiOel.1} and we define:
 	\begin{equation}\label{eq.def.U}
 	U_f(t) :=  C  + \frac{\mu_0}{\mu_1} \begin{cases}
 		\int_0^t \nrm{f}{L^\infty(s,T)} ds & t > 0 \\
 		t \nrm{f}{L^\infty(0,T)}& \text{otherwise}
 	\end{cases}
 	\end{equation}
 	and we complete this profile by adding an initial layer, 
	defining $S_f(t) := U(t) +\frac{ \|f\|_\infty }{\mu_1} w(t)$, where we denote 
	$\|f\|_\infty := \|f\|_{L^\infty(0,T)}$. 
	We compute 
 	\begin{align}
 		&S_f(t) - \int_0^t S_f(t-a) \tkernel(a)\, da \notag\\
 		&= \int_0^\infty \left(U_f(t) - U_f(t-a)\right) \tkernel (a)\, da
		+ \int_t^\infty U_f(t-a) \tkernel(a)\, da \notag\\
		&\quad +\frac{ \|f\|_\infty  }{\mu_1}  \int_t^\infty (a-t) \tkernel(a)\, da,
 	\end{align}
 	by the same arguments as in the proof of \cite[Theorem 1.1]{MiOel.1}, one writes:
	$$
	\begin{aligned}
	& \left(\int_0^t + \int_t^\infty \right)\left( U(t)-U(t-a)\right) \tkernel(a)\,da \\
	& \geq
	\frac{\mu_0}{\mu_1} \biggl\{\int_0^t \|f\|_{L^\infty(t,T)} a \tkernel(a)\, da \\
	& \quad + \int_t^\infty \bigl(t \|f\|_{L^\infty(t,T)}  + (a-t) \|f\|_{L^\infty(0,T)}\bigr) \tkernel(a)\, da \biggr\}
	\geq |f(t)|.
	\end{aligned}
	$$
	it remains to take care of the non-positive tail $\int_t^\infty U_f(t-a) \tkernel(a) da$ added in order to obtain the latter inequality.	
	one concludes that
 	\begin{align}
 	&S_f (t) - \int_0^t S_f(t-a) \tkernel(a)\, da \notag\\
	&\geq | f(t) | + \int_t^\infty \left(C+ \frac{\nrm{f}{L^\infty(0,T)} }{\mu_1} (t-a) \right)\tkernel(a)\, da \notag\\
	&\quad +  \frac{ \|f\|_\infty  }{\mu_1}  \int_t^\infty (a-t) \tkernel(a)\, da  \geq | f(t) |
 	\end{align}
 	the major improvement is here the compensation that our initial layer corrector $w$ provides: the linear part of $U$'s tail is negative and exactly 
 	matching the positive contribution from $w$. 
	As in the proof of Theorem \ref{thm.w},  the resolvent being positive one can apply \cite[Lemma 9.8.2 p. 257]{grip} in order to conclude. 
	Namely 
	$$
 	\begin{aligned}
 	|z(t)| - \int_0^t & |z(t-a)| \tkernel(a) da \\
	&  \leq \frac{|f(t)| }{\mu_0}
	+ \int_t^\infty \left|  \zp(t-a) \right|\tkernel(a)\, da =: \tilde{f} \\
	& \leq S_{\tilde f} (t)- \int_0^t S_{\tilde f}(t-a) \tkernel(a)\, da
 	\end{aligned}
 	$$
 	which implies that $|z(t)| \leq S_{\tilde{f}}(t)$, as soon as $|z(0)| \leq S(0)$, the latter condition is fulfilled by tuning the value of $C$ in the definition of $U$ in \eqref{eq.def.U}.
 \end{proof}
 
 \begin{remark}
 	The condition $-\da \kernel (a) / \kernel(a) \leq \zeta$  allows to consider polynomial decrease of the kernel at infinity. The method of recovering the tails introduced in \cite{MiOel.1}
 	in Lemma 2.4 fails for the latter kernels. If one considers only $U$ as a super-solution, then applying the integral operator implies:
 	$$
 	\begin{aligned}
 		U(t) - &\int_0^t U(t-a) \tkernel(a) da \geq |f(t)| + C \int_t^\infty \tkernel(a) da + \frac{\|f\|_\infty}{\mu_1} \int_t^\infty (t-a) \tkernel(a) da \\
 		& =  |f(t)| + \int_t^\infty \tkernel(a) da \left(C + \frac{\|f\|_\infty}{\mu_1} \frac{\int_t^\infty (t-a) \tkernel(a) da}{\int_t^\infty  \tkernel(a) da} \right)  \\
 			& = :
 			|f(t)| + \int_t^\infty \tkernel(a) da \left(C - \frac{\|f\|_\infty}{\mu_1} A [\tkernel](t) \right) 
 	\end{aligned}
 	$$

	Here one could ask that $A[\tkernel](t)$ is uniformly bounded in time, as in  \cite[Lemma 2.4]{MiOel.1}. However, there are simple counterexamples for this, such as $\kernel(a) = (1+a)^{-4}$ for which $A[\rho]$ is not uniformly bounded:
 	$$
 	A[\kernel](t) = \frac{\int_0^\infty a \kernel(a+t)da }{\int_0^\infty \kernel(a+t)da} \sim \frac{t}{2}.
 	$$
 	
	The motivation of our approach is that there is no need anymore 
to compute the bound on this latter quantity since the negative term of the right hand side in the comparison principle is compensated by the contribution of the initial layer corrector.

 	Another argument should be used in the case of a compactly supported kernel. Indeed in this case there exists a time $t_0$ great enough s.t.
 	$\int_0^{t} \tkernel(a) da = 1$, which prevents from using \cite[Lemma 9.8.2]{grip}. 
 \end{remark}

 \section{Numerical counterpart}\label{sec:discrete}
 In order to apply the previous ideas to the discrete setting, we first need to
 introduce a discrete framework that approximates the solutions of \eqref{eq.z}.
 We use piecewise constant approximations. Indeed, for any locally integrable function $g$ defined on $\rr$, we set:
 $ g^n := \frac{1}{h} \int_{t_n}^{t_{n+1}} g(t) dt$,  $\forall n \in \N$.
 Using the rectangular rule, the discrete counterpart of \eqref{eq.z} reads: find $(Z^n)_{n\in\{0, \ldots, N\}}$ solving 
 \begin{equation}\label{eq.disc.z}
 	\left\{\begin{aligned}
 		&	Z^n - h \sum_{j\geq 0} Z^{n-j} \tr_j = f^n , & n \in \{0, \ldots, N\} \\
 		&	Z^n = Z_p^n , & n<0
 	\end{aligned}\right.
 \end{equation}
 where the discrete kernel $(\tr_j)_{j\in\N}$ is defined as in Section \ref{subsec:discrete} below.		
The sequence $(\tr_j)_{j\in\N}$ discretizes the kernel $\tkernel$:
 $
 R_j := \frac{1}{h} \int_{jh}^{(j+1)h} \kernel(a) da 
 $.

\begin{lemma}
	Under the assumptions of Theorem \ref{thm.comparison}, the discrete kernel $(\tr_j)_{j\in\N}$ is a non-increasing sequence satisfying
	$ - \frac{\tr_j - \tr_{j-1}}{h \tr_j} \leq \zeta'$ for all $j \in \N^*$ and all $h \leq h_0$,
	where $\zeta' := 2\zeta$ and $h_0 := 1/(2\zeta)$.
\end{lemma}

\begin{proof}
	Writing the difference: $R_{j+1}-R_j = \frac{1}{h} \int_{jh}^{(j+1)h} (\kernel(a+h) - \kernel(a)) da$,
	then using the monotonicity of $\kernel$ provides that $R_{j+1} \leq R_j$ for all $j\in\N$.
	Next, using the hypothesis $\da \kernel(s) \geq -\zeta \kernel(s)$ and
	the monotonicity of $\kernel$ (which gives $\kernel(s) \leq \kernel(a)$ for $s\geq a$):
	\begin{align}
				R_{j+1}- R_{j} &= \frac{1}{h} \int_{jh}^{(j+1)h} (\kernel(a+h) - \kernel(a)) da
				= \frac{1}{h} \int_{jh}^{(j+1)h} \int_a^{a+h} \da \kernel(s)\, ds\, da \\
				&\geq - \zeta \frac{1}{h} \int_{jh}^{(j+1)h} \int_a^{a+h} \kernel(s)\, ds\, da
				\geq - \zeta \frac{1}{h} \int_{jh}^{(j+1)h} h\, \kernel(a)\, da
				= -\zeta h R_{j}
	\end{align}
	Rearranging: $(R_j - R_{j+1})/(h R_{j+1}) \leq \zeta/(1-\zeta h) \leq 2\zeta$ for $\zeta h \leq 1/2$.
	Setting $j \to j-1$ provides the desired result (with $\zeta$ replaced by $\zeta' := 2\zeta$).			
\end{proof}

 \subsection{The initial layer}\label{subsec:discrete}
We solve:
 \begin{equation}\label{eq.disc.w}
 	\left\{\begin{aligned}
 		&	 W^n - h\sum_{j\geq 0} W^{n-j} \tr_j =0 & n \geq 0 \\
 		& W^n = - nh & n< 0.
 	\end{aligned}
 	\right.
 \end{equation} 
 The discrete elongation variable can be defined as in the continuous setting: $U^n_j := W^n - W^{n-j}$, and the analog of \eqref{eq.elongation.u} becomes
 \begin{equation}\label{eq.disc.elongation}
 	U^{n+1}_{j+1} = U^n_j + \delta W^{n+\nud} ,\quad \delta W^{n+\nud} := W^{n+1}-W^n, \quad j\geq0,\quad n\geq 0.
 \end{equation}
Expressing \eqref{eq.disc.w} with the elongation variable leads to the conservation principle $h \sum_{j\ge0} U^n_j \tr_j = 0$ for all $n\ge0$. 
		
\begin{lemma}
	The discrete finite differences satisfy:
	$$
	\delta W^{n+\nud} = - h \sum_{j\geq 1} {(\tr_j - \tr_{j-1})} U^{n+1}_j,\quad n\geq 0
	$$
\end{lemma} 
 
\begin{proof}
	Summing \eqref{eq.disc.elongation} against $h \tr_j$ for $j\geq 0$  
	provides: 
	$h \sum_{j\geq 0} \tr_j U^{n+1}_{j+1} = h \sum_{j\geq 0} \tr_j U^n_j + \delta W^{n+\nud}$, 
	then adding and subtracting $h \sum_{j\geq 0} \tr_j U^{n+1}_j$, one obtains:
	\begin{equation}\label{eq.elo.dw}
		h \sum_{j\geq 0} \tr_j U^{n+1}_j + h \sum_{j\geq 0} \tr_j (U^{n+1}_{j+1}- U^{n+1}_j) = h \sum_{j\geq 0} \tr_j U^n_j + \delta W^{n+\nud}.
	\end{equation}
	Performing an integration by parts transforms the second term on the left hand side into:
	$$
	h \sum_{j\geq 0} \tr_j (U^{n+1}_{j+1}- U^{n+1}_j)  = - h \sum_{j\geq 1} {(\tr_j - \tr_{j-1})} U^{n+1}_j
	$$
	and the conservation principle cancels the first terms on both sides of \eqref{eq.elo.dw}.
\end{proof} 
 
\begin{corollary}
	The elongation variable satisfies a closed equation:
	\begin{equation}\label{eq.closed.disc.elongation}
		U^{n+1}_{k+1} = U^n_k - h \sum_{j\geq 1} {(\tr_j - \tr_{j-1})} U^{n+1}_j, \quad k\geq 0 ,\quad n\geq 0
	\end{equation}
	which is the discrete version of \eqref{eq.elongation.u}.
\end{corollary} 
 
\begin{proposition}\label{prop.W.bounded}
	Assuming that there exists $\zeta'>0$ s.t. for all $j \in \N^*$,  
	$$
	- \frac{R_j-R_{j-1}}{h R_{j}}  \leq \zeta' 
	$$
	The discrete initial layer $(W^n)_{n\in\N}$ solving \eqref{eq.disc.w} is a bounded non-negative sequence (and the bound is uniform with respect to $h$).
\end{proposition} 
 
\begin{proof}[Proof of Proposition \ref{prop.W.bounded}]
 	We set $P^n_j:= h \sum_{\ell=0}^n U^\ell_j$, and we notice again that, as $\sum_{j\geq 0}\tr_j U^{\ell}_j \equiv 0$, 
	$\sum_{j\geq 0} P^n_j \tr_j \equiv 0$. Next using \eqref{eq.closed.disc.elongation}, one obtains that:
 	$$
 	P^{n+1}_{k+1} - h U^0_{k+1} = P^n_k - h\sum_{j\geq 1} (\tr_j-\tr_{j-1})  (P^{n+1}_j-h U^0_j), \quad k\geq 0 
 	$$
 	which is the discrete equivalent of \eqref{eq.p}. 
	Now we want to avoid constant source terms in the left hand side of 
	the previous expression so that we define $\tP_j^n := P^n_j - h \sum_{k\in\{0,\dots,j\}} U^0_k$  
	which transforms the previous equality into:
 	$$
 	\tP_{k+1}^{n+1} = \tP_{k}^n - h \sum_{j\geq 1} (\tr_j-\tr_{j-1}) 
	 \left(\tP^{n+1}_j+ h \sum_{\ell = 0}^{j-1}U^0_\ell\right)
 	$$
 	But $\sum_{j\geq 1}( \tr_j-\tr_{j-1}) \sum_{\ell=0}^{j-1} U^0_{\ell} 
	= - \sum_{j\geq 1} \tr_j U^0_j - \tr_0 U^0_0 = 0$, leading to 
 	$$
 	\tP_{k+1}^{n+1} = \tP_{k}^n - h \sum_{j\geq 1} (\tr_j-\tr_{j-1})  \tP^{n+1}_j , \quad \forall k \geq 0, \; \forall n\geq 0.
 	$$

 	Applying the absolute value to the previous expression and integrating against $h \tr_k$ yields:
 	$$
 	I^n:= h \sum_{k\geq 0} \tr_k |\tP^{n+1}_{k+1}|  \leq h \sum_{k\geq 0} \tr_k |\tP_k^n|   
	- h \sum_{j\geq 1} (\tr_j-\tr_{j-1}) | \tP_j^n |  =: J^n
 	$$ 
 	But 
 	$I^n = h \sum_{k\geq 0} \tr_k |\tP^{n+1}_{k}| - h \sum_{k\geq 1}( \tr_k-\tr_{k-1}) |\tP^{n+1}_{k}| $, 
	and as $I^n \leq J^n$, this simplifies into:
 	$$
 	h \sum_{k\geq 0} \tr_k |\tP^{n+1}_{k}| \leq h \sum_{k\geq 0} \tr_k |\tP^{n}_{k}| 
 	$$
 	and by induction this provides:
 	$$
 	h \sum_{k\geq 0} \tr_k |\tP^{n+1}_{k}|  
 	\leq h \sum_{k\geq 0} \tr_k |\tP^{0}_{k}| 
 	\leq  h \sum_{k\geq 0} \tr_k \left|  h \sum_{\ell=0}^k U^0_\ell \right| 
 	\leq C(\mu_{2,h}, \mu_{1,h},\mu_0) 
 	$$
 	Returning to the definition of $(W^n)_{n\in\N}$,  one has 
 	$$
 	\delta W^{n+\nud} = - h \sum_{j\geq 1} (\tr_j-\tr_{j-1})U^{n+1}_j
 	$$
 	leading to 
 	$$
 	W^{n+1}-W^0 = -  \sum_{j\geq 1} (\tr_j-\tr_{j-1})(P^{n+1}_j- h U^0_j)
 	$$
 	and then 
 	$$
 \begin{aligned}
 		| W^{n+1}| \leq&  | W^0 | - h \sum_{j\geq 1} 
 		\frac{\tr_j-\tr_{j-1}}{h \tr_j} \tr_j \left( | P^{n+1}_j | +h | U_j^0 |\right) \\
 	& \leq | W^0 | + \zeta' \left\{ h \sum_{j\geq 0}  \tr_j  | P^{n+1}_j | +h^2 \sum_{j\geq 0}  \tr_j    | U_j^0 |\right\}\\
 	& \leq | W^0 | + h \zeta' \left\{ \sum_{j\geq 0}  \tr_j  | \tP^{n+1}_j | +h \sum_{j\geq 0}  \tr_j    \sum_{k=0}^{j} \left| U_k^0 \right| + c_1 \right\} \\
 	& \leq | W^0 | + h \zeta' \left\{ 2 h \sum_{j\geq 0}  \tr_j    \sum_{k=0}^{j} \left| U_k^0 \right| + c_1 \right\} 	
 	\leq C(\mu_{0,h},\mu_{1,h},\mu_{2,h}) 
 \end{aligned}
 	$$
Next we prove that $W^n \geq 0$ for all $n\geq 0$ by induction. 
First $W^0=\mu_{1,h}/(1-h \tr_0) \geq 0$ for $h$ 
small enough (and the bound is uniform by Lebesgue's Theorem for $h<h_0$). 
By definition of $W$, 
$W^{n} =  -nh \geq 0$ for $n<0$. 
Assume now that $W^k \geq 0$ for all $k \leq n$, then using \eqref{eq.disc.w}, 
$
(1-h\tr_0) W^{n+1} = h \sum_{j\geq 1} \tr_j W^{n+1-j} \geq 0
$, 
which concludes the proof.
 \end{proof}
 
 \subsection{Discrete resolvent and comparison principle}

First we discretise \eqref{eq.z}. To do so we recall that the meshsize is $h>0$ and the discrete kernel reads $(R_j)_{j\in\N}$ as in the previous section.

 By using successive convolutions of $(\tr_j)_{j\in\N}$, one can construct a non-negative $(Q_j)_{j\in\N}$ resolvent that solves 
 $$
 Q_n - h \sum_{j=0}^n Q_{n-j} \tr_j = \tr_n, \quad \forall n \geq  0.
 $$
 The discrete convolution of two sequences $(U_j)_{j\in\N}$ and $(V_j)_{j\in\N}$ is defined as:
 $$
 (U \star V)_n := h \sum_{j=0}^n U_{n-j} V_j , \quad \forall n \geq 0.
 $$
 The solution of \eqref{eq.disc.z} can be expressed as $Z = f + f \star Q$.

\begin{theorem}\label{thm.comp.disc}
	Assume that $-(R_j-R_{j-1})/(hR_{j}) \leq \zeta'$ for all $h>0$ and all $j\geq 1$, 
	and assumptions \ref{hypo.data}, then the solution of \eqref{eq.disc.z} 
	$(Z^n)_{n\in \{0,\dots,N\}}$ is bounded uniformly with respect to $h$.
\end{theorem}

 \begin{proof}[Proof of Theorem \ref{thm.comp.disc}
	(stated in Section \ref{sec:assumptions} as Theorem \ref{thm.comp.disc.main})]
 	First notice that thanks to the first hypothesis of the claim, 
	$h\sum_{j=0}^n \tr_j = 1 - h\sum_{j=n+1}^\infty  \tr_j \leq 1 - (1+\zeta' h)^{-n+1} \tr_0/\zeta' < 1$.
 	Then let's define $Q^m := \sum_{k=0}^m R^{k\star}$,  in the sense of sequences in $\ell^1$ such that 
	$(R\star R)_n = h\sum_{j=0}^n R_{n-j} R_j$ . 
 	One has that $Q = \lim_{m\to \infty} Q^m$ remains non-negative and finite for any $n\geq 0$. 
 	Indeed 
 	$|Q^m_n| \leq \sum_{k= 0}^m (h \sum_{j=0}^n R_j )^k \leq  \sum_{k= 0}^m  (1-\delta(n))^k \leq \delta(n)^{-1} < \infty$. 
	The solution $Z$ can then be expressed as $Z=f+f \star Q$ 
 	and as the verification is purely algebraic it follows the same steps as in \cite[Theorem 2.3.5 p. 44]{grip}. 
 	In the same spirit, one can extend \cite[Lemma 9.8.3]{grip} and show that if $S \geq S\star \tr + f$ 
	in the sense of convolution of sequences, then $S^n\geq Z^n$  for every $n\geq 0$, 
	where $Z$ is the sequence solving the equation $Z=\tr \star Z +f$.  

	The purpose is now to construct $S$. We set 
 	$$
 	S^n_f := \frac{|f|_\infty}{\mu_{1,h}} W^n + V^n ,
	$$
	where
	$$
	V^n:=  C + \frac{1}{\mu_{1,h}}  \begin{cases}
 		h \sum_{j=0}^n \max_{k \in \{ j,\dots,N\}} |f^k| & \text{if } n \geq 0, \\
 		h n |f|_\infty & \text{otherwise.}
 	\end{cases}
 	$$
 	then applying the integral operator gives: 
 	$$
 	\begin{aligned}
 		I^n &:=  S^n_f - h \sum_{j=0}^{n} S^{n-j}_f \tr_j \\
		& \geq h \sum_{j=1}^n (V^n-V^{n-j})\tr_j
 				+ h \sum_{j=n+1}^\infty (V^n-V^{n-j})\tr_j
 				+  h \sum_{j= n+1}^\infty V^{n-j} \tr_j \\
 				& \qquad
 				+ h \frac{|f|_\infty}{\mu_{1,h}} \sum_{j=n+1}^\infty h(j-n) \tr_j\\
 		& \geq \frac{h}{\mu_{1,h}} \sum_{j=1}^n h \!\sum_{k=n-j}^n \max_{\ell \in \{ k,\dots,N\}} |f^\ell| \tr_j \\
		& \qquad + \frac{h}{\mu_{1,h}} \sum_{j=n+1}^\infty \sum_{k=0}^n h \max_{\ell \in \{ k,\dots,N\}} |f^\ell| \tr_j
			+ \sum_{j=n+1}^\infty (V^0-V^{n-j} )\tr_j  \\
 		& \qquad+  h \sum_{j= n+1}^\infty V^{n-j} \tr_j  + h \frac{|f|_\infty}{\mu_{1,h}} \sum_{j=n+1}^\infty h(j-n) \tr_j\\
 		& \geq
		\frac{h}{\mu_{1,h}} \sum_{j=1}^n h j |f^n| \tr_j
		+ \frac{h}{\mu_{1,h}} \sum_{j=n+1}^\infty h n |f^n| \tr_j \\
		& \qquad + \frac{|f|_\infty}{\mu_{1,h}}\sum_{j=n+1}^\infty h({j-n} )\tr_j
 		+  h \sum_{j= n+1}^\infty V^{n-j} \tr_j \\
		& \qquad + h \frac{|f|_\infty}{\mu_{1,h}} \sum_{j=n+1}^\infty h(j-n) \tr_j \\
		& \geq |f^n|\\
 	\end{aligned}
 	$$
 	The first line in the last inequality is then greater than $|f^n|$ (the three summations with respect to $j$ reduce to 1) 
	and the second line cancels since the initial layer corrector's right hand side compensates exactly the
 	 tail of $U$ that was added and subtracted in the first lines above. One then compares:
 	 $$
 	 \left|Z^n\right|- h \sum_{j=0}^n \left| Z^{n-j} \right| \tr_j 
	 \leq 
	 |\tf^n|  \leq S_{\tf} - h \sum_{j=0}^n S^{n-j}_{\tf} \tr_j
 	 $$
 	where $\tf^n := f^n + h \sum_{j=n+1}^\infty |Z^{n-j}| \tr_j$, which implies that $|Z^n| \leq S^n$ as soon as $C \geq |Z^0|$.
 \end{proof}

\section{Application  to delayed sweeping processes}\label{sec:application}

We now apply the comparison principle developed in the previous sections to study delayed sweeping processes.
Recall Hypothesis \ref{hypo.convex.main} on the family of closed convex sets.

In this section we solve the following discrete problem: for any $n\geq 0$,
 \begin{equation}\label{eq.sweeping.projection}
	 X^n := P_{C^n}(\overline{X}^n), \quad \overline{X}^n := \frac{h \sum_{j\geq 1} X^{n-j} \tr_j}{h \sum_{j\geq 1} \tr_j}.
\end{equation}
where $P_{C^n}$ is the projection onto the convex set $C^n$ defined above. 

This problem can be reformulated as the following minimization problem: for any $n\geq 0$,
\begin{equation}\label{eq.sweeping}
	 X^n := \argmin_{W \in C^n} \cE_n(W), \quad \cE_n(W) := \frac{h}{2}\sum_{j\geq1} (W-X^{n-j})^2 \tr_j.
\end{equation}
 
\begin{proposition}[Minimization principle]
	The two minimizing problems \eqref{eq.sweeping.projection} and \eqref{eq.sweeping} are equivalent.
\end{proposition} 
 
\begin{proof}
	By definition of the projection onto a convex set: given $\overline{X}^n \in \mathbb{R}^d$,
	there exists a unique $X^n \in C^n$ s.t.
	$$
	X^n = \argmin_{W \in C^n} | W - \overline{X}^n |^2 
	$$
	Now we start from the energy functional $\cE_n$ and we compute:
	$$
	\begin{aligned}
		\cE_n(W) & = \frac{h}{2} \sum_{j\geq 1} 
		\left( W - \overline{X}^n + \overline{X}^n - X^{n-j} \right)^2 \tr_j \\
		& = \frac{h}{2} \sum_{j\geq 1} \left( W - \overline{X}^n \right)^2 \tr_j
		+ h \sum_{j\geq 1} (W - \overline{X}^n)(\overline{X}^n - X^{n-j}) \tr_j \\
		& \quad + \frac{h}{2} \sum_{j\geq 1} (\overline{X}^n - X^{n-j})^2 \tr_j  \\
		& = \frac{1}{2}(1-h \tr_0)| W - \overline{X}^n |^2 +  R((X^{n-j})_{j\geq 1})
	\end{aligned}
	$$
	where the remainder term $R$ does not depend on $W$. This concludes the proof.
\end{proof}

 \subsection{Energy estimates}
 
\begin{proposition}\label{prop.nrj}
	Assume that the sequence $(\tr_j)_{j\in \N}$ is decreasing and that the convex sets satisfy
	Hypothesis \ref{hypo.convex.main}. 
	Then the discrete sweeping sequence $(X^n)_{n\ge0}$ defined by \eqref{eq.sweeping}	 
	satisfies the energy estimate:	
	\[
	\cE_{N}(X^{N}) +  \sum_{n=0}^{N-1} D_n \le e^{2 T} \cE_{0}(X^0) + e^{2 T} \sum_{n=0}^{N-1} e_n,
	\]
	where $D_n= \frac{h}{2} \sum_{j\ge1} (\tr_{j}-\tr_{j+1}) | X^{n} - X^{n-j}|^2 \geq 0$ is the dissipation term and
	$e_n:= e(C^n,C^{n+1})^2/h$.
\end{proposition}
 
 \begin{proof}[Proof of Proposition \ref{prop.nrj}, stated in Section \ref{sec:assumptions} as Proposition \ref{prop.nrj.main}]
 	The minimization principle \eqref{eq.sweeping} gives
 	\[
 	\cE_{n+1}(X^{n+1}) \le \cE_{n+1}(P_{C^{n+1}}(X^n)).
 	\]
 	Adding and subtracting $X^n$ inside the energy functional leads to
 	\[
 	\begin{aligned}
 		\cE_{n+1}(X^{n+1})
 		& \le \frac{1-h\tr_0}{2}
 		| X^n - P_{C^{n+1}}(X^n)|^2 \\
		& \quad + h\sum_{j\ge1}( X^n - P_{C^{n+1}}(X^n),\,X^{n+1-j}-X^n)\tr_j \\
 		& \quad + \frac{h}{2} \sum_{j\geq 1} | X^{n+1-j} - X^n|^2 \tr_j \\
 	\end{aligned}
 	\]
 	We denote $e_n:= e(C^n,C^{n+1})^2/h$, and write accordingly
 	\[
 	\begin{aligned}
 		\cE_{n+1}(X^{n+1})
 		& \le \frac{h}{2} e_n
 		+ h \sum_{j\ge1} e(C^n,C^{n+1})|X^{n+1-j}-X^n|\tr_j \\
 		& \quad +  \frac{h}{2} \sum_{j\geq 1} | X^{n+1-j} - X^n|^2 \tr_j .
 	\end{aligned}
 	\] 
 	Using Young's inequality on the second term on the right-hand side gives
 	\[
 	\begin{aligned}	
 		\cE_{n+1}(X^{n+1})
 		& \le (1+2 h ) e_n 
 		+ (1+2 h) \cE_{n+1}(X^n).		
 	\end{aligned}
 	\]
 	Then using the definition of the energy functional and rearranging terms leads to
 	\[	
 	\cE_{n+1}(X^n) = \cE_{n}(X^n) - \frac{h}{2} \sum_{j\ge1} (\tr_{j+1}-\tr_j) | X^{n} - X^{n-j}|^2 =: \cE_{n}(X^n) - D_n,
 	\]
 	where $D_n= \frac{h}{2} \sum_{j\ge1} (\tr_{j}-\tr_{j+1}) | X^{n} - X^{n-j}|^2 \geq 0$ is the dissipation term.
 	Putting everything together, we obtain
 	\[
 	\cE_{n+1}(X^{n+1}) + (1+2 h) D_n \le (1+2 h) \cE_{n}(X^n) + (1+2 h ) e_n.
 	\]
 	Iterating this inequality from $n=0$ to $n=N-1$ gives
 	\[
 	\cE_{N}(X^{N}) + (1+2 h) \sum_{n=0}^{N-1} D_n \le (1+2 h)^N \cE_{0}(X^0) + (1+2 h ) \sum_{n=0}^{N-1} (1+2 h)^{N-1-n} e_n.
 	\]
 	Since $(1+2 h)^N \le e^{2 T}$ and $\sum_{n=0}^{N-1} (1+2 h)^{N-1-n} e_n \le e^{2 T} \sum_{n=0}^{N-1} e_n$, the proof is complete.	
 \end{proof}
 
 \begin{remark}
 	The energy $\cE_{N}(X^N)$ completes the dissipation term at step $n=N$ since one writes:
 	$$
 	\begin{aligned}
	h \sum_{p=0}^N & \sum_{j\geq1} (\tr_j - \tr_{j+1}) \left| X^{p}- X^{p-j}\right|^2 \\
	& \leq 2 \cE_{N}(X^{N})
	+ 2 h\sum_{p=0}^{N-1} \sum_{j\geq1} (\tr_j - \tr_{j+1}) \left| X^{p}- X^{p-j}\right|^2
	\end{aligned}
 	$$
 \end{remark}
 
\begin{corollary}\label{coro.dissip}
	Under the previous hypotheses, 
	$$
	\sum_{p= 0}^N h 
	\left( \sum_{j\ge1} (\tr_{j-1} - \tr_{j}) \left| X^{p} - X^{p-j}\right| \right)^2 \leq  C < \infty
	$$
	uniformly with respect to the discretization step $h$.
\end{corollary} 
 
\begin{proof}
	Since the sequence $(\tr_{j})_{j\geq 0}$ is non-increasing and of mass 1 (indeed $h \sum_{j\geq 0} \tr_{j}=1$), one has that
	$$
	\sum_{j\ge1} \tr_{j-1} - \tr_{j} = \tr_0 \leq (h \Lip_\kernel/2  + \kernel(0))/\mu_{0} < \infty
	$$
 where we have used the Lipschitz hypothesis in Assumptions \ref{hypo.data}.
	Moreover, using the discrete ratio bound $(\tr_{j-1}-\tr_j) \leq \zeta' h \tr_j$ for $j\geq 1$,
	one has
	$$
	h \sum_{j\geq 1} (\tr_{j-1} -\tr_{j}) |X^p-X^{p-j}|^2
	\leq \zeta' h^2 \sum_{j\geq 1} \tr_j |X^p-X^{p-j}|^2 = 2\zeta' \cE_p(X^p).
	$$
	Then using Cauchy-Schwarz,
	one writes:
	$$
	\begin{aligned}
		\sum_{p= 0}^N & h \left( \sum_{j\geq1} (\tr_{j-1} -\tr_{j}) |X^p-X^{p-j}| \right)^2 \\
		& \leq \sum_{p=0}^N h \sum_{j\geq 1} (\tr_{j-1} -\tr_{j})
								\sum_{j\geq 1} (\tr_{j-1} -\tr_{j}) |X^p-X^{p-j}|^2\\
		& \leq  \tr_0 \sum_{p= 0}^N h \sum_{j\geq 1}  (\tr_{j-1} -\tr_{j}) |X^p-X^{p-j}|^2
		\leq  2 \zeta' \tr_0 \sum_{p= 0}^N  \cE_p(X^p)  < \infty
	\end{aligned}
	$$
	where the last bound follows from Proposition \ref{prop.nrj}.
\end{proof}

\begin{corollary}\label{coro.lag}
	Under the previous hypotheses, for $h$ small enough, one has the uniform bound:
	$$
	\left|\overline{X}^n - X^n\right| \leq C 
	$$
\end{corollary}

\begin{proof}
	Using Cauchy-Schwarz inequality, one writes:
	$$
	\begin{aligned}
		\left|\overline{X}^n - X^n\right|^2
		& = \left| \frac{h \sum_{j\geq 1} (X^{n-j} - X^n) \tr_j}{1-h \tr_0} \right|^2 \\
		& \leq  \frac{h \sum_{j\geq 1} (X^{n-j}-X^n)^2 \tr_j }{1-h \tr_0} \times \frac{h \sum_{j\geq 1} \tr_j}{1-h \tr_0}
		\leq \frac{2 \cE_n(X^n)}{1-h \tr_0} < \infty
	\end{aligned}
	$$
	where we have used the energy estimate provided by Proposition \ref{prop.nrj}.	
\end{proof}
\begin{remark}
	We notice that contrarily to the standard sweeping process without delay, 
	the distance between the projected point $X^n$ and the unconstrained point $\overline{X}^n$ 
	remains only bounded whereas in the standard case it is $h$ close. 
	This greatly complicates the 
	analysis and does not allow to obtain straightforward Lipschitz or $BV$ estimates in time as in the classical cases \cite{mor77}. 
\end{remark}

\subsection{The case of a moving circular set}\label{sec:circular}
 
In this section we prove Theorem \ref{thm:circular.main}. We consider the case where the convex sets are moving circles:
 \begin{equation}
 	C(t):= \overline{B}(c(t), R) = \{x \in \mathbb{R}^m:  \phi_t(x):= |x - c(t)|^2 -	 R^2 \leq 0\}
 \end{equation}
 with $c \in H^1([0,T];\mathbb{R}^m)$ and $R \in \mathbb{R}_+$.
We discretize the sets as $C^n:= C(nh)$ for $n=0,\dots,N$ with $N = \lfloor T/h \rfloor$.
The corresponding discrete sweeping process is defined by \eqref{eq.sweeping}. Under these assumptions,
Hypothesis \ref{hypo.convex.main} is satisfied since
 $$ d_H(C^{n},C^{n-1}) \leq |c(nh)-c((n-1)h)| $$ 
 and $c$ being in $H^1([0,T];\mathbb{R}^m)$ implies that 
 $$
 \sum_{n=1}^{N} d_H(C^{n},C^{n-1})^2/h \leq \sum_{n=1}^{N} |c(nh)-c((n-1)h)|^2/h \leq |c|_{H^1([0,T];\mathbb{R}^m)}^2 < \infty.
 $$
We denote by $\varphi_n(X):= |X - c^n|^2 - R^2$ where $c^n:= c(nh)$. These constraints
are qualified in the sense of \cite{Ciarlet89} since $\nabla \varphi_n(X) = 2 (X - c^n) \neq 0$ 
for all $X$ such that $\varphi_n(X) =0$.	This provides the necessary condition
in order to  write Euler Lagrange equations (in their Kuhn and Tucker version)  
associated to the minimization problem \eqref{eq.sweeping}, they read:
 \begin{equation}\label{eq.sweep.eul}
 	h \sum_{j\geq 0} (X^n - X^{n-j})\tr_j  + \lambda^n \nabla \varphi_n (X^n)=0, \quad \forall n\geq 0.
 \end{equation}
Since $\mu_0 = 1$, the Euler-Lagrange equation identifies $\lambda^n \nabla \varphi_n(X^n) = \overline{X}^n - X^n$.
We denote the \emph{projection displacement increment}
$$\delta P^{n+\nud} := \lambda^{n+1} \nabla \varphi_{n+1}(X^{n+1}) - \lambda^{n} \nabla \varphi_{n}(X^{n})$$
and recall the \emph{step} $\delta X^{n+\nud} := X^{n+1} - X^{n}$.

\begin{proof}[Proof of Theorem \ref{thm:circular.main}]
 	We recall the definition of $\varphi_n(X):= |X - c^n|^2 - R^2$ where $c^n:= c(nh)$.
 	From the Euler-Lagrange equation \eqref{eq.sweep.eul}, setting the elongation variable as 
 	$U^n_j:= X^n-X^{n-j}$ for $j\geq 1$, it  satisfies by definition 
 	$$
 	U^{n+1}_{j+1} = U^n_j + X^{n+1} - X^n 
 	$$
 	which summed against $\tr_j$ gives
 	\begin{equation}\label{eq:inner.prod}
 		h \sum_{j\geq 0} U^{n+1}_{j+1} \tr_j = h \sum_{j\geq 0} U^n_j \tr_j + \mu_{0} (X^{n+1}-X^n) = 
 		h \sum	_{j\geq 0} U^n_j \tr_j + X^{n+1}-X^n , 
 	\end{equation}
 	since $\mu_0 = h \sum_j \tr_j =1$. 
 	Adding and subtracting $U^{n+1}_j$ in the left hand side gives
 	$$
 	\begin{aligned}
	h \sum_{j\geq 0} U^{n+1}_{j+1} \tr_j
	& = h \sum_{j\geq 0} U^{n+1}_j \tr_j + h \sum_{j\geq 0} (U^{n+1}_{j+1}-U^{n+1}_j) \tr_j \\
 	& = h \sum_{j\geq 1} U^{n}_j \tr_j + (X^{n+1}-X^n).
	\end{aligned}
 	$$
 	By a change of indices one recovers that
	$$\sum_{j\geq 0} (U^{n+1}_{j+1}-U^{n+1}_j) \tr_j = \sum_{j\geq 1}
 	U^{n+1}_j(\tr_{j-1}-\tr_{j}) =: d^{n+1}.$$
	We use Cauchy-Schwarz inequality to estimate $d^{n+1}$:
	$$
	\begin{aligned}
		(d^{n+1})^2 & \leq
		h \sum_{j\geq 1} |U^{n+1}_j|^2 \tr_j \cdot
		h \sum_{j\geq 1} \left(\frac{R_{j-1}-R_j}{h}\right)^2 \frac{1}{\tr_j} \\
		& \leq 2 \cE_{n+1}(X^{n+1}) C_R =: C_R'
	\end{aligned}
	$$
	where we have used the definition of the energy functional $\cE_{n+1}$ and hypothesis \eqref{eq.kernel.estimate.circular}. 
	Here we underline that although $d^{n+1}$ seems to be related to the dissipation term of the energy estimates from Proposition \ref{prop.nrj},
 the index shift prevents from using Corollary \ref{coro.dissip} to obtain a uniform bound on $d^{n+1}$, 
 and we were forced to rely on the alternative above.

 	The Euler-Lagrange equation \eqref{eq.sweep.eul},
 	can be rephrased as 
 	$$
 	h \sum_{j\ge 0 } U^{n}_j \tr_{j}+ \lambda^n G^n = 0
 	$$
 	and provides using the equality:
 	$$
 	-\lambda^{n+1} G^{n+1} +  d^{n+1} =  -\lambda^n G^n + (X^{n+1}-X^n)
 	$$	
 	where we defined $G^n:= \nabla \varphi_n (X^n) = 2(X^n - c^n)$.
	Testing the previous equality by $X^{n+1}-X^n$ gives:
 	$$
 	|X^{n+1}-X^n|^2 + \left( \lambda^{n+1}  G^{n+1} - \lambda^n G^n, X^{n+1}-X^n \right) 
 	=  h \left( d^{n+1}, X^{n+1}-X^n \right).
 	$$
	The second term on the left-hand side is the projection-step product
	$$\langle \delta P^{n+\nud}, \delta X^{n+\nud} \rangle,$$
	for which we need a lower bound (see Remark~\ref{rem:projection_step} and Figure~\ref{fig.circle.fig8}).
 	We decompose it as:
 	$$
 	\begin{aligned}
 		2 & \left(\lambda^{n+1}  G^{n+1} - \lambda^n G^n, X^{n+1}-X^n \right) \\
		& = \left(\lambda^{n+1}  G^{n+1} - \lambda^n G^n, G^{n+1}-G^n \right) \\
		& \quad + 2 \left(\lambda^{n+1}  G^{n+1} - \lambda^n G^n, c^{n+1}-c^n \right) \\
 		& = \frac{1}{2} \left( \lambda^{n+1}+ \lambda^{n}\right) | G^{n+1}-G^n |^2
 		+ \frac{1}{2} (\lambda^{n+1}- \lambda^{n}) ((G^{n+1})^2 - (G^n)^2) \\
 		& \quad +
 		2 \left(\lambda^{n+1}  G^{n+1} - \lambda^n G^n, c^{n+1}-c^n \right)	 \\
 		& = \frac{1}{2} \left( \lambda^{n+1}+ \lambda^{n}\right) | G^{n+1}-G^n |^2 
 		+ 2 (\lambda^{n+1}- \lambda^{n}) (\varphi_{n+1}(X^{n+1}) - \varphi_n(X^n)) \\
 		& \quad +
 		2\left(\lambda^{n+1}  G^{n+1} - \lambda^n G^n, c^{n+1}-c^n \right)
 	\end{aligned}		
 	$$
 	Notice first that since $(\lambda^n)_{n\in \N}$ is non-negative, the first term is non-negative and we can neglect it in the following estimates. 
 	We focus now on the second term. 
 	\begin{itemize}[itemsep=0.05cm,parsep=0.1cm]
 		\item Assume that the constraint is active at step $n$ and $n+1$, i.e. $\lambda^n >0$ and $\lambda^{n+1} >0$. 
 		Then $\varphi_n(X^n) = 0$ and $\varphi_{n+1}(X^{n+1})= 0$ and the second term vanishes.
 		\item Assume that the constraint is inactive at step $n$ and $n+1$, i.e. $\lambda^n =0$ and $\lambda^{n+1} =0$. 
 		Then the second term vanishes as well.
 		\item Assume that the constraint is active at step $n$ and inactive at step $n+1$, i.e. 
 		$\lambda^n >0$, $\varphi_n(X^n) = 0$	 and $\lambda^{n+1} =0$, $\varphi_{n+1}(X^{n+1}) < 0$,
 		the second term reduces to $-\lambda^n \varphi_{n+1}(X^{n+1}) \geq 0$.
 		\item Assume that the constraint is inactive at step $n$ and active at step $n+1$, i.e. 
 		$\lambda^n =0$, $\varphi_n(X^n) < 0$ and $\lambda^{n+1} >0$, $\varphi_{n+1}(X^{n+1}) = 0$.
 		the second term reduces to $\lambda^{n+1} (-\varphi_n(X^n)) \geq 0$.
 	\end{itemize}

	In all cases, the second term is non-negative. To sum up, one has
 	$$
 	\begin{aligned}
 		(X^{n+1}-X^n)^2 & \leq - \langle \lambda^{n+1} G^{n+1} - \lambda^n G^n, c^{n+1}-c^n \rangle 
 		+  h \langle d^{n+1}, X^{n+1}-X^n \rangle
 	\end{aligned}	
 	$$
 	Using then that $\lambda^{n+1} G^{n+1} - \lambda^n G^n = - (X^{n+1}-X^n) + h \sum_{j\geq 0} (X^{n+1-j}-X^{n-j}) \tr_j$
 	from \eqref{eq.sweep.eul}, 
 	denoting $\delta X^{n+\nud} = X^{n+1}-X^n$, we have	using Young's inequality:
 	$$
 	\begin{aligned}			
 		|\delta X^{n+\nud}|^2 & \leq \left( \delta X^{n+\nud}, c^{n+1}-c^n \right) \\
		& \quad - \left( h \sum_{j\geq 0} \delta {X}^{n-j+\nud} \tr_j, c^{n+1}-c^n \right)
		+ h \left( d^{n+1}, \delta X^{n+\nud} \right) \\
 		& \leq |\delta X^{n+\nud}| |c^{n+1}-c^n|
		+ h \left|\sum_{j\geq 0} \delta {X}^{n-j+\nud} \tr_j\right| |c^{n+1}-c^n| \\
 		& \quad +  h |d^{n+1}| |\delta X^{n+\nud}| \\
 		& \leq \alpha_1 |\delta X^{n+\nud}|^2 + \alpha_1^{-1} |\delta c^{n+\nud}|^2 \\
		& \quad + \alpha_2 \left|h \sum_{j\geq 0} \delta {X}^{n-j+\nud} \tr_j\right|^2
		+ \alpha_2^{-1} |c^{n+1}-c^n|^2 \\
		& \quad + h^2  \alpha_3^{-1} |d^{n+1}|^2 +  \alpha_3 |\delta X^{n+\nud}|^2
 	\end{aligned}
 	$$
 	Leading to the final estimate:
 	$$
 	\begin{aligned}			
 		|\delta X^{n+\nud}|^2 &\leq 
 		(\alpha_1^{-1} + \alpha_2^{-1}) /(1 - \alpha_1 - \alpha_3)  |c^{n+1}-c^n|^2 \\
 		& + \alpha_2/(1 - \alpha_1 - \alpha_3)  \left|h \sum_{j\geq 0} 
		\delta {X}^{n-j+\nud} \tr_j\right|^2 + h^2	 \alpha_3^{-1} C_R'	/(1 - \alpha_1 - \alpha_3	) 		
 	\end{aligned}
 	$$
 	One chooses $\alpha_2=1-\alpha_1 - \alpha_3$ and dividing the previous expression by $h$ gives:
 	$$
 	\begin{aligned}			
 		y^n:= {} & \frac{|\delta X^{n+\nud}|^2}{h} \\
		& \leq
 		\frac{\alpha_1^{-1} + (1-\alpha_1 - \alpha_3)^{-1}}{1 - \alpha_1 - \alpha_3}  \frac{|c^{n+1}-c^n|^2}{h}
		+ \frac{1}{h} \left|h \sum_{j\geq 0} \delta {X}^{n-j+\nud} \tr_j\right|^2 \\
		& \quad + \frac{\alpha_3^{-1} h}{1 - \alpha_1 - \alpha_3} C_R'	\\
 		& \leq c \omega_n  + h \sum_{j\geq 0} y^{n-j} \tr_j 	
 	\end{aligned}
 	$$
 	where $\omega_n:= (c^{n+1}-c^n)^2/h + h $ and we used Jensen's inequality to 
 	estimate $\left|h \sum_{j\geq 0} \delta {X}^{n-j+\nud} \tr_j\right|^2$.
 	Define now $z^n:= \sum_{k=0}^{n} y^k$, then summing the previous inequality gives:
 	$$
 	z^n \leq c \sum_{k=0}^n \omega_k + h \sum_{j=0}^{n} \tr_j z^{n-j}
 	$$
 	We detail here how to obtain the last sum: for $n \geq 1$ 
 	$$
 	\begin{aligned}
 		\sum_{p= 0}^{n}\sum_{j\geq 0} y^{p-j} \tr_j
		& = \sum_{p=0}^{n} \sum_{j=0}^p y^{p-j} \tr_j = \sum_{0\leq j \leq p \leq n} y^{p-j} \tr_j \\
		& = \sum_{j= 0}^{n} \sum_{k=0}^{n-j} y^{k} \tr_j
 		= \sum_{j=0}^{n} z^{n-j} \tr_j
 	\end{aligned}
 	$$
 	since we assumed that the data is well prepared, for all $j \geq 1$ $X^{-j} = X^{-1}\in C^0$ implying that $y^k = 0$ for $k < 0$.

 	Using then Theorem \ref{thm.comp.disc.main}  leads to the desired estimate since 
 	$\sum_{k=0}^n \omega_k \leq C$. Indeed, 
	since $\dt c \in L^2 (0,T)$, its equivalent 
 	on the discrete level provides $\sum_{n=0}^{N_T} (\delta c^{n+ \nud})^2/h < \infty$. This ends the proof.	
 \end{proof}
 
\begin{remark}\label{rem:projection_step}
	The key step of the proof relies on bounding the projection-step product
	$\langle \delta P^{n+\nud}, \delta X^{n+\nud} \rangle$.
	Without time dependency, this term would be non-negative by monotonicity of the normal cone.
	Here, however, $X^{n+1}$ and $X^n$ do not necessarily belong to the same set, and the decomposition above
	yields a positive part and a component depending on the center trajectories.
	This heavily depends on the geometry of the moving set, and extending
	these results to general convex sets --- for instance as a function of the excess $e(C^{n},C^{n+1})$ --- remains an open problem.
	Numerical simulations (Figure~\ref{fig.circle.fig8}) suggest that $-\langle \delta P^{n+\nud}, \delta X^{n+\nud} \rangle / h^2$ remains bounded as $h \to 0$, which is consistent with the convergence of the scheme.
\end{remark}

 \subsection{Convergence towards the delayed sweeping process}

 \begin{proof}[Proof of Theorem \ref{thm:limit_sweeping.main}]
 	\emph{Step~1 -- Compactness of $\{X_h\}_h$.}
 	From the uniform $H^1$ bound, we obtain (Kondrachov’s selection theorem) that there exists
 	$X\in H^1([0,T];\mathbb R^m)$ and a subsequence such that
 	$X_h\to X$ in $C^0(0,T)$.
	Moreover, the piecewise-constant interpolant $\tilde{X}_h$ also converges to $X$ in $L^p(0,T)$.
 	
 	\emph{Step~2 -- Convergence of the discrete averages.}
	By Proposition \ref{prop:discrete_convolution_convergence}, 
	the discrete convolution defining $\overline{X}_h$ converges to the continuous convolution $\overline{X}$
	in $L^p$ for any $1\le p<\infty$. The continuous limit convolution decomposes accordingly as
 	\[
 	\overline{X}(t)
 	:=\int_0^t X(t-s)\varrho(s)\,ds
 	+ X_p \int_t^\infty \varrho(s)\,ds.
 	\]
 	
 	\emph{Step~3 -- Convergence in the discrete inclusion.}
 	At the discrete level, for each $n\ge0$, we have the projection characterization $
 	X^n = P_{C^n}(\overline{X}^n) $ is equivalent to the normal cone inclusion
 	$\overline{X}^n - X^n \in N_{C^n}(X^n)$ or, in duality form,
 	\[
	\forall \xi \in \R^m, \langle \overline{X}^n - X^n , \xi \rangle \leq \left| \overline{X}^n - X^n \right| d_{C^n}(\xi + \overline{X}^n - X^n).
 	\]
 	Define the piecewise-constant sets $C_h(t):=C^n$ for $t\in(nh,(n+1)h]$.
 	By the $H^1(0,T)$ assumption on the centers of balls  $c(t)$,
 	we may assume $C_h(t)\to C(t)$ in the Hausdorff sense for a.e.\ $t$.
 	Moreover by Corollary~\ref{coro.lag}, we have that $X^n-\overline{X}^n$ is uniformly bounded in $\R^m$.

	Up to a subsequence, we may fix $t\in(0,T)$ such that
 	$\overline{X}_h(t)\to \overline{X}(t)$, $\tilde{X}_h(t)\to X(t)$,
 	and $C_h(t)\to C(t)$.
 	Then for such $t$,
 	\[
 	\langle \overline{X}_h(t) - \tilde{X}_h(t) , \xi \rangle 
 	\leq \left| \overline{X}_h(t) - \tilde{X}_h(t) \right| d_{C_h(t)}(\xi + \overline{X}_h(t) - \tilde{X}_h(t)).
 	\]
	Because one has that 
	$$
	\begin{aligned}
		\left| d_{C_h(t)} (\xi + \right. & \overline{X}_h(t)- \tilde{X}_h(t))  \left.- d_{C(t)}(\xi + \overline{X}(t) - X(t))\right| \\
	& \leq \left| (\overline{X}_h(t) - \tilde{X}_h(t))- (\overline{X}(t) - X(t)) \right| + d_H(C_h(t),C(t)),
	\end{aligned}
	$$
	and that $d_{C_h(t)} (\xi + \overline{X}_h(t) - \tilde{X}_h(t))$ is uniformly bounded for every fixed $\xi$ 
	thanks to the boundedness of $\overline{X}_h(t) - \tilde{X}_h(t)$ (see Corollary~\ref{coro.lag}),
 	we obtain, by successive triangle inequalities and passing to the limit $h\to0$,
 	\[
 	\langle \overline{X}(t) - X(t) , \xi \rangle 
 	\leq \left| \overline{X}(t) - X(t) \right| d_{C(t)}(\xi + \overline{X}(t) - X(t)).
 	\]
 	Since this holds for all $\xi\in\R^m$, which is equivalent to the normal cone inclusion
 	$
 	\overline{X}(t)-X(t)\in N_{C(t)}\big(X(t)\big)
 	$.	
 \end{proof}

\section{Numerical implementation and examples}\label{sec:numerics}

We now describe the numerical algorithm used to compute the discrete sweeping sequence.
We consider kernels of the form
$\kernel(a) = \varepsilon e^{-\varepsilon a}$ with $\varepsilon > 0$
and feasible sets
$C(t) = \{x \in \RR^d : g(x - c(t)) \geq 0\}$,
where $c(t)$ is a given trajectory and $g$ is a smooth concave function with $g > 0$.
An interactive simulation platform is available at \url{https://steffenpl.github.io/delayed-sweeping} (source code: \url{https://github.com/SteffenPL/delayed-sweeping}).

\subsection{Time-stepping scheme}

For $n = 0, 1, \ldots, N$ with $N = \lfloor T/h \rfloor$:
\begin{compactenum}[1)]
	\item Compute the discrete kernel weights
	$R_j = \frac{1}{h}\int_{jh}^{(j+1)h} \kernel(a)\, da$,
	truncated at $j = J_{\max}$ where $R_{J_{\max}} < \texttt{tol}$.
	For the exponential kernel, $R_j = \frac{1}{h}(1 - e^{-\varepsilon h}) e^{-\varepsilon j h}$.
	\item Compute the weighted average of past positions
	$$\overline{X}^n = \frac{h \sum_{j=1}^{J_{\max}} R_j\, X^{n-j}}{h \sum_{j=1}^{J_{\max}} R_j},$$
	where $X^{n-j} = X_p((n-j)h)$ whenever $n-j < 0$.
	\item Compute the projection $X^n = P_{C^n}(\overline{X}^n)$ as described below.
\end{compactenum}

\subsection{Projection onto the constraint set}

When $g_n(\overline{X}^n) \geq 0$, the point is feasible and $X^n = \overline{X}^n$.
Otherwise we must compute $P_{C^n}(\overline{X}^n)$, the closest point on $\{g_n = 0\}$.
All gradients are approximated by central finite differences with step $\delta = 10^{-6}$.

\paragraph{Step 1: Newton to boundary.}
Starting from $p_0 = \overline{X}^n$, we iterate
\begin{equation}\label{eq.newton}
	p_{k+1} = p_k - \frac{g_n(p_k)}{|\nabla g_n(p_k)|^2}\, \nabla g_n(p_k)
\end{equation}
until $|g_n(p_k)| < \texttt{tol}$. This produces a point $p^{(0)}$ on the boundary,
but not necessarily the closest one to $\overline{X}^n$.

\paragraph{Step 2: Optimality refinement.}
At the metric projection $x^*$, the displacement $x^* - \overline{X}^n$ must be
parallel to $\nabla g_n(x^*)$. Starting from $p^{(0)}$, for $\ell = 0, 1, \ldots$
we compute the tangential component of the displacement
$d^{(\ell)} = p^{(\ell)} - \overline{X}^n$:
\begin{equation}\label{eq.tangential}
	d_\tau^{(\ell)} = d^{(\ell)} - \frac{d^{(\ell)} \cdot \nabla g_n(p^{(\ell)})}{|\nabla g_n(p^{(\ell)})|^2}\, \nabla g_n(p^{(\ell)}).
\end{equation}
If $|d_\tau^{(\ell)}| < \texttt{tol}$, we accept $X^n = p^{(\ell)}$.
Otherwise, we remove the tangential component and re-project onto the boundary:
\begin{equation}\label{eq.refinement}
	\hat{p}^{(\ell)} = p^{(\ell)} - \alpha\, d_\tau^{(\ell)}, \qquad
	p^{(\ell+1)} = \text{Newton-to-boundary}\bigl(\hat{p}^{(\ell)}\bigr),
\end{equation}
where $\alpha \in (0,1]$ is chosen by backtracking to ensure
$|p^{(\ell+1)} - \overline{X}^n| < |p^{(\ell)} - \overline{X}^n|$.
In practice, $\texttt{tol} = 10^{-8}$ and at most $50$ outer iterations suffice.

\subsection{Numerical examples}\label{sec:numerical_examples}

All numerical experiments use the discrete time-stepping scheme~\eqref{eq.sweeping} with the exponential kernel $\varrho(a) = \varepsilon e^{-\varepsilon a}$.

\paragraph{Circular constraint (Figures~\ref{fig.simulations}--\ref{fig.circle.fig8}).}
The constraint is a disk $C^n = \{X \in \mathbb{R}^2 : |X - c^n| \leq R\}$ whose center follows a Lissajous curve:
\[
c(t) = \bigl(2\sin(t),\; 2\sin(2t)\bigr), \quad t \in [0, 9].
\]
Parameters: $R = 0.5$, $\varepsilon = 0.75$, $h = 0.005$, with zero past condition $X_p(t) = (0, 0)$ for $t < 0$.
This past is compatible with the initial constraint since $X_p(0) = (0,0) \in C^0$.

\paragraph{Stadium constraint (Figures~\ref{fig.simulations}--\ref{fig.circle.fig8}).}
The constraint is a rotating stadium (capsule) defined by
$C(t) = \{x \in \mathbb{R}^2 : r - \sqrt{\max(|x_1| - R/2,\, 0)^2 + x_2^2} \geq 0\}$,
translated along the same Lissajous curve and rotated with angle $\alpha(t) = 4t$.
Parameters: $R = 0.72$, $r = 0.3$, $\varepsilon = 0.75$, $h = 0.005$, with zero past condition.

\subsection{Convergence and projection-step product}\label{sec:convergence_tests}

For both examples, we study convergence as $h \to 0$ with step sizes $h \in \{2^{-2}, 2^{-3}, \ldots, 2^{-8}\}$ and reference solution $h_{\mathrm{ref}} = 2^{-10}$.
The $L^2$ error is computed as
$\|z^h-z^{\mathrm{ref}}\|_{L^2}^2 \approx h\sum_{k=1}^{N} \|z^h(hk)-z^{\mathrm{ref}}(hk)\|^2$.
With compatible past ($X_p \equiv 0 \in C^0$), we observe second-order convergence for both constraint types (Figure~\ref{fig:conv_l2_loglog}).
With incompatible past ($X_p \equiv (2,0) \notin C^0$), the convergence degenerates to first order.

As discussed in Remark~\ref{rem:projection_step}, the projection-step product $\langle \delta P^{n+\nud}, \delta X^{n+\nud} \rangle$ plays a central role in the energy estimates.
Figure~\ref{fig.circle.fig8} shows that the rescaled quantity $-\langle \delta P^{n+\nud}, \delta X^{n+\nud} \rangle / h^2$ remains bounded as $h \to 0$, for both the circular and the stadium constraint.
This is consistent with the convergence of the scheme and supports the bound used in the proof of Theorem~\ref{thm:circular.main}.

\begin{figure}[htbp]
\centering
\begin{subfigure}[t]{0.48\textwidth}
	\centering
	\begin{tikzpicture}
	\begin{loglogaxis}[
		width=0.93\textwidth,
		xlabel={$h$},
		ylabel={$\|z^h-z^{\mathrm{ref}}\|_{L^2}$},
		grid=major,
		legend style={
			font=\small,
			at={(0.5,1.05)},
			anchor=south,
			legend columns=2,
		},
		mark size=2pt,
	]
	\addplot[blue, thick, mark=*] table[
		col sep=tab, x=h, y=Error
	]{fig_convergence-L2-circle.tsv};
	\addlegendentry{Compatible}
	\addplot[orange, thick, mark=square*] table[
		col sep=tab, x=h, y=Error
	]{fig_convergence-L2-circle-incompat.tsv};
	\addlegendentry{Incompatible}
	\addplot[red, thick, dashed] table[
		col sep=tab,
		x expr=\thisrow{h},
		y expr=0.2*(\thisrow{h})^2
	]{fig_convergence-L2-circle.tsv};
	\addlegendentry{$h^2$}
	\addplot[black, thick, dashdotted] table[
		col sep=tab,
		x expr=\thisrow{h},
		y expr=0.2*(\thisrow{h})
	]{fig_convergence-L2-circle.tsv};
	\addlegendentry{$h$}
	\end{loglogaxis}
	\end{tikzpicture}
	\caption{Circular constraint}
	\label{fig:conv_l2_circle}
\end{subfigure}
\hfill
\begin{subfigure}[t]{0.48\textwidth}
	\centering
	\begin{tikzpicture}
	\begin{loglogaxis}[
		width=0.93\textwidth,
		xlabel={$h$},
		ylabel={$\|z^h-z^{\mathrm{ref}}\|_{L^2}$},
		grid=major,
		legend style={
			font=\small,
			at={(0.5,1.05)},
			anchor=south,
			legend columns=2,
		},
		mark size=2pt,
	]
	\addplot[blue, thick, mark=*] table[
		col sep=tab, x=h, y=Error
	]{fig_convergence-L2-stadium.tsv};
	\addlegendentry{Compatible}
	\addplot[orange, thick, mark=square*] table[
		col sep=tab, x=h, y=Error
	]{fig_convergence-L2-stadium-incompat.tsv};
	\addlegendentry{Incompatible}
	\addplot[red, thick, dashed] table[
		col sep=tab,
		x expr=\thisrow{h},
		y expr=0.5*(\thisrow{h})^2
	]{fig_convergence-L2-stadium.tsv};
	\addlegendentry{$h^2$}
	\addplot[black, thick, dashdotted] table[
		col sep=tab,
		x expr=\thisrow{h},
		y expr=0.2*(\thisrow{h})
	]{fig_convergence-L2-stadium.tsv};
	\addlegendentry{$h$}
	\end{loglogaxis}
	\end{tikzpicture}
	\caption{Stadium constraint}
	\label{fig:conv_l2_stadium}
\end{subfigure}
\caption{$L^2$ convergence of the discrete solution. Compatible past ($X_p \equiv 0 \in C^0$, blue) shows second-order convergence; incompatible past ($X_p \equiv (2,0) \notin C^0$, orange) degenerates to first order. Reference slopes $h$ and $h^2$ are shown.}
\label{fig:conv_l2_loglog}
\end{figure}
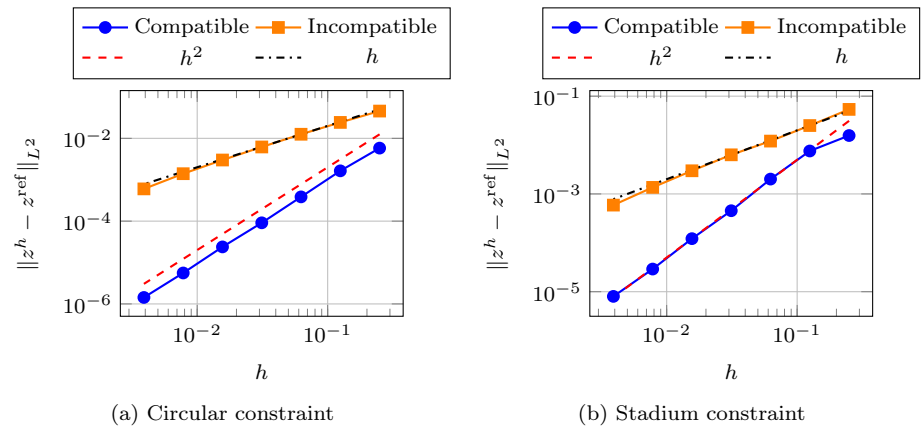

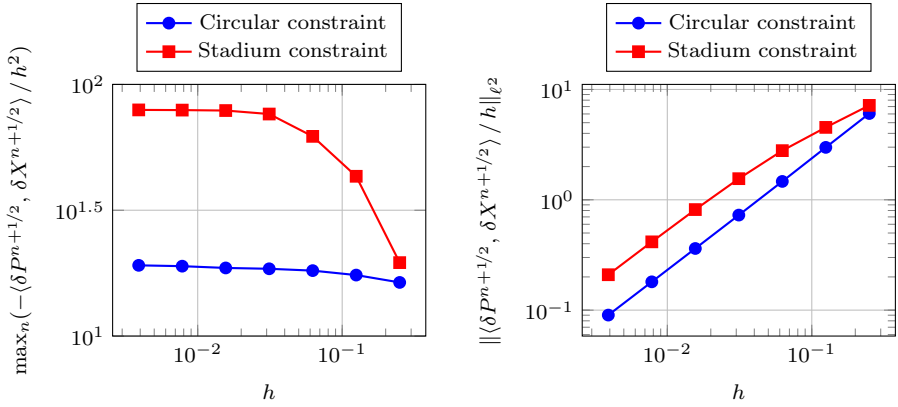
\begin{figure}
	\centering
	\begin{subfigure}[t]{0.48\textwidth}
		\centering
		\begin{tikzpicture}
			\begin{loglogaxis}[
				width=\textwidth,
				xlabel={$h$},
				ylabel={$\max_n (-\langle \delta P^{n+\nud},\, \delta X^{n+\nud} \rangle \,/\, h^2)$},
				ylabel style={font=\small},
				ymin=10, ymax=100,
				legend style={
					font=\small,
					at={(0.5,1.05)},
					anchor=south,
					legend columns=1,
				},
				grid=major,
				mark size=2pt,
			]
			\addplot[mark=*, thick, blue]
				table[col sep=tab, x=h, y={max KKT}]
				{fig_projection-step-product-max-circle.tsv};
			\addlegendentry{Circular constraint}
			\addplot[mark=square*, thick, red]
				table[col sep=tab, x=h, y={max KKT}]
				{fig_projection-step-product-max-stadium.tsv};
			\addlegendentry{Stadium constraint}
			\end{loglogaxis}
		\end{tikzpicture}
		\caption{$\max_n (-\langle \delta P^{n+\nud}, \delta X^{n+\nud} \rangle / h^2)$ vs.\ $h$}
		\label{fig.dPdX.max}
	\end{subfigure}
	\hfill
	\begin{subfigure}[t]{0.48\textwidth}
		\centering
		\begin{tikzpicture}
			\begin{loglogaxis}[
				width=\textwidth,
				xlabel={$h$},
				ylabel={$\| \langle \delta P^{n+\nud},\, \delta X^{n+\nud} \rangle \,/\, h \|_{\ell^2}$},
				ylabel style={font=\small},
				legend style={
					font=\small,
					at={(0.5,1.05)},
					anchor=south,
					legend columns=1,
				},
				grid=major,
				mark size=2pt,
			]
			\addplot[mark=*, thick, blue]
				table[col sep=tab, x=h, y={max KKT}]
				{fig_projection-step-product-L2-circle.tsv};
			\addlegendentry{Circular constraint}
			\addplot[mark=square*, thick, red]
				table[col sep=tab, x=h, y={max KKT}]
				{fig_projection-step-product-L2-stadium.tsv};
			\addlegendentry{Stadium constraint}
			\end{loglogaxis}
		\end{tikzpicture}
		\caption{$\|\langle \delta P^{n+\nud}, \delta X^{n+\nud} \rangle / h\|_{\ell^2}$ vs.\ $h$}
		\label{fig.dPdX.L2}
	\end{subfigure}
	\caption{Projection-step product $\langle \delta P^{n+\nud}, \delta X^{n+\nud} \rangle$ for varying step size~$h$, comparing circular and stadium-shaped constraints. See Section~\ref{sec:numerical_examples} for full specifications.}
	\label{fig.circle.fig8}
\end{figure}

\section{Conclusion and perspectives}

This work provides a novel type of comparison principle and applies this principle to the analysis 
of a discrete sweeping process with memory. Our comparison principle is flexible enough to yield 
compactness and convergence results for the discrete sweeping process. But the non-infinitesimal 
nature of the projections imposes strong geometric constraints. Further work is needed to extend 
our results to general convex sets and for applications in mathematical biology eventually also 
for prox-regular constraints.

Our focus in this work was to be able to start from a delayed constrained gradient flow \cite{Mi.5}
and show convergence of the discrete sweeping process to its continuous inclusion.
An alternative route to more general results is to consider Yoshida approximations
and fixed point techniques directly in the continuous setting. This is an ongoing work. 
The original model introducing the specific delayed gradient flow was motivated by 
the modeling of cell migration in structured environments \cite{OelzSch10},  
and contained a parameter $\varepsilon$ representing the inverse of the bond's stiffness
and the characteristic lifetime of the bonds. For the sake of clarity we did not include this
parameter in the present work. The limit $\varepsilon \to 0$ is another perspective of interest since in this case the delayed sweeping process converges to the classical Moreau process.
These aspects will be handled in a forthcoming publication.

On the numerical side, natural extensions include the multicellular case, 
a rigorous numerical analysis of the observed convergence rates and 
the role of incompatible initial past conditions, and the use of continuation methods to 
accelerate the projection step.

\section{Acknowledgements}

S.P. was supported by the Japanese Society for the Promotion of Science (JSPS) KAKENHI Grant Number JP22K13971. This manuscript was edited with help of AI assistance, and the computational code was created using Claude Code. The mathematical research and proofs were developed by the authors without AI assistance.

 \appendix

	 \section{Discrete approximation of kernels' moments}\label{sec:moments}
	
	In order to compare the continuous quantities 
	\(\mu_k:= \int_0^{\infty} a^k \varrho(a)\,da\) 
	with their discrete counterparts that appear in the estimates above,
	we introduce for every \(h>0\) the discrete moments
	\[
	\mu_{k,h}:= h \sum_{j\ge0} (jh)^k R_j,
	\qquad 
	R_j:= \frac{1}{h}\int_{jh}^{(j+1)h}\varrho(a)\,da,
	\qquad k\in\{0,1,2\}.
	\]
	Note that by construction the discrete mass is exact:
	\[
	\mu_{0,h}=h\sum_{j\ge0}R_j=\int_0^\infty\varrho(a)\,da=\mu_0.
	\]
	The following lemma shows that these quantities are first–order accurate
	approximations of the continuous moments under our integrability assumptions.
	
	\begin{lemma}[Approximation of the moments]
		\label{lem:moments}
		Assume that the kernel satisfies \(\varrho\ge0\) and 
		\(\int_0^{\infty}(1+a)^{k-1}\varrho(a)\,da<\infty\) 
		for \(k\in\{1,2\}\). Then there exist constants \(C_1,C_2>0\),
		independent of \(h\in(0,1]\), such that
		\begin{equation}\label{eq:moment_error}
			|\mu_k-\mu_{k,h}|
			\;\le\;
			C_k\,h
			\int_0^{\infty}(1+a)^{k-1}\varrho(a)\,da,
			\qquad
			C_1=1,\; C_2=3.
		\end{equation}
	\end{lemma}
	
	\begin{proof}
		For every \(k\ge1\),
		\[
		\mu_k
		=\sum_{j\ge0}\int_{jh}^{(j+1)h}a^k\varrho(a)\,da.
		\]
		For \(a\in[jh,(j+1)h]\), by the mean value theorem,
		\(
		|a^k-(jh)^k|
		\le k (jh+h)^{k-1} h.
		\)
		Hence
		\[
		|\mu_k-\mu_{k,h}|
		\le
		k h \sum_{j\ge0}(jh+h)^{k-1}
		\int_{jh}^{(j+1)h}\varrho(a)\,da
		\le
		C_k h\int_0^{\infty}(1+a)^{k-1}\varrho(a)\,da,
		\]
		where \(C_k=k2^{k-1}\), giving the claimed constants
		\(C_1=1\) and \(C_2=3\).
	\end{proof}

	\begin{remark}
		Estimates \eqref{eq:moment_error}
		justify the substitution of discrete moments
		\(\mu_{i,h}\) for their continuous counterparts
		\(\mu_i\) in all bounds of discrete {\em a priori} estimates above. 
		In particular, the constants appearing in those inequalities
		remain uniform in \(h\) as \(h\to0\).
	\end{remark}

\section{$L^p$ Convergence of Discrete Convolutions}\label{sec:discrete_convolution}

Let $T > 0$ and $h = T/N$ for $N \in \mathbb{N}$. Define the intervals $J_n = [nh, (n+1)h)$ for $n=0, \dots, N-1$. 
Consider the piecewise constant functions:
\[ B_h(t) = \sum_{n=0}^{N-1} B_n \chi_{J_n}(t), \quad C_h(t) = \sum_{n=0}^{N-1} C_n \chi_{J_n}(t) \]

\begin{proposition}\label{prop:discrete_convolution_convergence}
Assume $B_h \to B$ strongly in $L^p(0,T)$ and $C_h \to C$ strongly in $L^1(0,T)$. 
Define the discrete causal convolution $A_h$ as:
\[ A_h(t) = \sum_{n=0}^{N-1} \left( h \sum_{j=0}^n B_{n-j} C_j \right) \chi_{J_n}(t) \]
Then $A_h \to B * C$ strongly in $L^p(0,T)$, where $(B * C)(t) = \int_0^t B(t-s)C(s)ds$.
\end{proposition}

\begin{proof}
The proof proceeds in two main steps: identifying $A_h$ as a projection of a continuous convolution, and applying Young's Inequality.

\paragraph{Step 1: Identification of the Operator}
For piecewise constant functions $B_h$ and $C_h$, the continuous convolution $(B_h * C_h)(t)$ evaluated at a grid point $t = nh$ is:
\[ (B_h * C_h)(nh) = \int_0^{nh} B_h(nh-s)C_h(s)ds = \sum_{j=0}^{n-1} \int_{jh}^{(j+1)h} B_h(nh-s)C_h(s)ds \]
Since $B_h$ and $C_h$ are constant on each interval $J_j$, the integral becomes $h B_{n-1-j} C_j$. Thus, the discrete sum in $A_h$ corresponds exactly to the values of the continuous convolution of the piecewise constant approximations at the nodes. Specifically, $A_h$ is the piecewise constant interpolant of $(B_h * C_h)$.

\paragraph{Step 2: Decomposition of the Error}
By the triangle inequality in $L^p(0,T)$:
\[ \|A_h - B * C\|_{L^p} \leq \|A_h - B_h * C_h\|_{L^p} + \|B_h * C_h - B * C\|_{L^p} \]

The first term $\|A_h - B_h * C_h\|_{L^p}$ vanishes as $h \to 0$ because $B_h * C_h$ is a continuous (piecewise linear) function, and $A_h$ is its piecewise constant approximation. 

For the second term, we apply Young's Inequality for convolutions ($\|f * g\|_{L^p} \leq \|f\|_{L^p}\|g\|_{L^1}$):
\[ \|B_h * C_h - B * C\|_{L^p} \leq \|(B_h - B) * C_h\|_{L^p} + \|B * (C_h - C)\|_{L^p} \]
\[ \leq \|B_h - B\|_{L^p}\|C_h\|_{L^1} + \|B\|_{L^p}\|C_h - C\|_{L^1} \]

Since $B_h \to B$ in $L^p$, the term $\|B_h - B\|_{L^p} \to 0$. Since $C_h \to C$ in $L^1$, the term $\|C_h - C\|_{L^1} \to 0$. By the stability of strong convergence, $\|C_h\|_{L^1}$ is uniformly bounded. Therefore, both terms approach zero.
And the claim follows.
\end{proof}

\section{Decay of decreasing functions in weighted \(L^1\) spaces}\label{sec:decay_lipschitz}
\begin{proposition}\label{prop:decay_weighted}
Let $\kernel: \mathbb{R}_+ \to \mathbb{R}$ be a monotone non increasing non-negative function such that 
$\kernel \in L^1(\mathbb{R}_+, (1+a)^2)$. 
Then $\lim_{a \to \infty} \kernel(a)(1+a)^2 = 0$.
\end{proposition}

\begin{proof}
	Let $f(a) = \kernel(a)(1+a)^2$. By hypothesis, $f \in L^1(\mathbb{R}_+)$.
	Denote $f^n:= \frac{1}{h} \int_{nh}^{(n+1)h} f(a) da$ for $n \in \mathbb{N}$ and $h > 0$.
	Since $f$ is non-negative and integrable, we have $\lim_{n \to \infty} f^n = 0$.
	Moreover, since $\kernel$ is non-increasing,
	\[
	f^{n-1}\geq \kernel(nh) \frac{1}{h}  \int_{(n-1)h}^{nh} (1+a)^2 da \geq  
	\kernel(nh) (nh)^2.
	\]
	for any $n \geq 1$. Thus one has:
	$$
	(1+a)^2 \kernel(a) \leq 2 f^{n-1} \quad \text{ for } a \in [nh, (n+1)h]
	$$
	which shows the desired limit by letting $n \to \infty$.
\end{proof}

\bibliographystyle{spmpsci}
\bibliography{biblio}

\end{document}